\documentclass[a4paper,11pt,leqno]{amsart}
\usepackage{amsmath}
\usepackage{amsfonts}
\usepackage{mathtools}
\usepackage{amssymb}
\usepackage{bbm}
\usepackage{amscd}
\usepackage{color}
\usepackage[colorlinks=true,urlcolor=blue,citecolor=blue,linkcolor=blue,linktocpage,pdfpagelabels,bookmarksnumbered,bookmarksopen]{hyperref}
\usepackage[margin=1in]{geometry}
\usepackage{tikz}
\usetikzlibrary{positioning,arrows.meta,shapes.geometric,calc}
\theoremstyle{plain}
\numberwithin{equation}{section}
\newtheorem{theorem}{Theorem}[section]
\newtheorem{prop}[theorem]{Proposition}
\newtheorem{lem}[theorem]{Lemma}

\newtheorem{defn}[theorem]{Definition}

\theoremstyle{definition}
\newtheorem{rmk}[theorem]{Remark}

\def\div{\mathop {\rm div}\nolimits}

\newcommand{\R}{\mathbb{R}}

\makeatletter
\def\cleardoublepage{\clearpage\if@twoside \ifodd\c@page\else
\hbox{}
\thispagestyle{empty}
\newpage
\if@twocolumn\hbox{}\newpage\fi\fi\fi}
\makeatother

\title{$H$-compactness for nonlocal linear operators in fractional divergence form}

\author[M.\ Caponi]{Maicol Caponi}

\author[A.\ Carbotti]{Alessandro Carbotti}

\author[A.\ Maione]{Alberto Maione}

\address[M.\ Caponi]{Contributing Author\newline\indent Dipartimento di Ingegneria e Scienze dell'Informazione e Matematica \newline\indent Università degli Studi dell'Aquila,\newline\indent Via Vetoio 1, Coppito, 67100 L'Aquila, Italy}
\email{maicol.caponi@univaq.it}

\address[A.\ Carbotti]{Corresponding Author\newline\indent Dipartimento di Matematica e Fisica\newline\indent Università del Salento\newline\indent Via Per Arnesano, 73100 Lecce, Italy}
\email{alessandro.carbotti@unisalento.it}

\address[A.\ Maione]{Contributing Author\newline\indent Centre de Recerca Matemàtica\newline\indent Edifici C, Campus Bellaterra, 08193 Bellaterra, Spain}
\email{amaione@crm.cat}



\begin{document}


\begin{abstract}
We study the $H$-convergence of nonlocal linear operators in fractional divergence form, where the oscillations of the matrices are prescribed outside the reference domain. Our compactness argument bypasses the failure of the classical localisation techniques that mismatch with the nonlocal nature of the operators involved. 
If symmetry is also assumed, we extend the equivalence between the $H$-convergence of the operators and the $\Gamma$-convergence of the associated energies. 
\end{abstract}

\maketitle



\section{Introduction}

In the last decade there has been an increasing interest in extending the $H$-convergence theory for linear operators to a \textit{nonlocal} framework, motivated by the huge number of applications of nonlocal operators in fluid mechanics, image denoising, nonlocal minimal surfaces, nonlinear elasticity, anomalous diffusion, and stable Lévy processes (see e.g.~\cite{bucval16, cdv19} and the references therein). 
First attempts to deal with nonlocal $H$-convergence appear in~\cite{BE,FBRS17}, where the authors study scalar perturbations of the fractional $p$-Laplace operator, and in~\cite{BuaGhoGho19}, regarding the $H$-convergence of fractional powers of elliptic operators in divergence form. We also mention~\cite{BraBruDo,KPZ19,PatVal15} for the homogenisation of more general nonlocal energies.

All the aforementioned contributions in nonlocal $H$-convergence address only scalar weights, 
in contrast to the classical results~\cite{DGS,Spagnolo,Tartar}. 
On the other hand, the case of matrix weights is of great interest in applications, as it allows to study anisotropic heterogeneous materials.
We observe that it is not immediately clear how to get a momentum operator in problems involving Gagliardo seminorms, since the latter are defined via integration of scalar energies.

We hence wonder whether it might be possible to formulate a nonlocal $H$-convergence-type problem in terms of differential operators. 
To this aim, we exploit a suitable notion of fractional divergence $\div^s$ and fractional gradient $\nabla^s$, in a way that the fractional divergence acts on the nonlocal momentum $A(x)\nabla^s u(x)$. 
The construction of this class of fractional-order operators is explained in details in Section~\ref{sec:prel}, and relies on the pioneering contributions by Shieh and Spector~\cite{ShiehSpector,ShiehSpectorII}, \v{S}ilhav\'y~\cite{silhavy}, and the subsequent works~\cite{BCCS22,ComiStefani19,ComiStefani22,ComiStefani23}.
This leads to the study of the $H$-convergence of the following sequence of linear elliptic problems
\begin{equation*}\tag{$P_h^f$}
\begin{cases}
-\div^s(A_h\nabla^s u_h)=f&\text{in }\Omega, \\
u_h=0\quad&\text{in }\R^n\setminus\Omega,
\end{cases}
\end{equation*}
where $(A_h(x))_h$ is a sequence of (not necessarily symmetric) matrices, and $f$ is a given source term, in the same spirit as the works by De Giorgi-Spagnolo~\cite{DGS,Spagnolo} and Murat-Tartar~\cite{Tartar}.

When dealing with the asymptotic behaviour of problems $(P_h^f)$, two main challenges arise. 
The first concerns the localization techniques typically used in the $H$-convergence theory to identify the limit problem (see e.g.~\cite{Tartar}). A key idea to overcome this difficulty, introduced by Kreisbeck and collaborators~\cite{CuKrSc23, KrSc22}, is to use the Riesz potential $I_{1-s}$. This operator allows to reinterpret the (nonlocal) fractional gradient $\nabla^s u$ as the classical (local) gradient $\nabla w$ of the auxiliary variable $w \coloneqq I_{1-s} u$, and to rephrase $(P_h^f)$ in a local framework (see Proposition~\ref{prop:Riesz}). 
This approach was successfully used in~\cite{KrSc22} to address lower semicontinuity and relaxation of functionals involving fractional gradients, and in~\cite{ABSS23} to establish $\Gamma$-convergence results for Ginzburg–Landau-type energies.

In the present paper we apply this technique to the possibly non-variational linear problems $(P_h^f)$. In particular, special care is required when dealing with boundary-value problems, since the Riesz potential does not preserve boundary conditions.
 To overcome this issue, we exploit the compactness properties of the fractional Sobolev space $H^s_0(\Omega)$ and use a suitable \emph{fractional} Leibniz rule (see Subsection~\ref{sec:Riesz}).

The second question concerns the behaviour of the matrices $A_h(x)$ outside $\Omega$. 
Differently from the local scenario, since the fractional gradient $\nabla^s$ must be defined on the whole space $\R^n$, the fractional divergence $\div^s$ acts on vector fields globally defined over $\R^n$ (see Definition~\ref{def:divnablafraz}). 
Hence, the matrix-valued functions $A_h(x)$ must be defined on the whole space $\R^n$. 
On the other hand, $-\div^s(A_h\nabla^su_h)$ is bounded only in $\Omega$, and so we cannot hope to obtain compactness of $(A_h)_h$ outside $\Omega$. 
In order to tackle this lack of compactness, the values of any matrix $A_h(x)$ are prescribed outside $\Omega$ by a fixed matrix $A_0(x)$, which satisfies standard growth conditions. 

In details, we study the compactness, in the sense of the $H$-convergence, of the class
\[
\mathcal M(\lambda,\Lambda,\Omega,A_0)\coloneqq\left\{A\in\mathcal M(\lambda,\Lambda,\mathbb{R}^n):A=A_0\text{ a.e.\ in $\mathbb{R}^n\setminus\Omega$}\right\},
\]
for positive constants $\lambda\le\Lambda$ and $A_0\in\mathcal M(\lambda,\Lambda,\mathbb{R}^n)$, where $\mathcal M(\lambda,\Lambda,\mathbb{R}^n)$ denotes the class of matrix-valued measurable functions $A\colon \mathbb{R}^n\to\R^{n\times n}$ satisfying the growth conditions
\begin{align*}
 &A(x)\xi\cdot \xi\ge \lambda|\xi|^2& &\text{for all $\xi\in\mathbb{R}^n$ and for a.e.\ $x\in \mathbb{R}^n$},\\
 &A(x)\xi\cdot\xi \ge \Lambda^{-1}|A(x)\xi|^2& &\text{for all $\xi\in\R^n$ and for a.e.\ $x\in \mathbb{R}^n$}.
\end{align*}

Our first main result is Theorem~\ref{thm:compactness}, which is the $H$-compactness of the class $\mathcal M(\lambda,\Lambda,\Omega,A_0)$ with respect to a suitable notion of nonlocal $H$-convergence, see Definition~\ref{def:Hscon}. 
More precisely, we prove that any sequence $(A_h)_h$ in $\mathcal M(\lambda,\Lambda,\Omega,A_0)$ admits, up to a subsequence, a limit matrix-valued function $A_\infty$ such that $(A_h)_h$ $H$-converges to $A_\infty$. 
Moreover, $A_\infty$ still belongs to $\mathcal M(\lambda,\Lambda,\Omega,A_0)$, leading to the compactness of the class.
We also show that the $H$-limit $A_\infty$ coincides with its local counterpart on the reference domain $\Omega$, thus revealing a consistency between local and nonlocal $H$-convergence.
As a by-product, in Theorem~\ref{thm:sym-compact} we show that the subclass of symmetric matrices, here denoted by $\mathcal M^{\rm sym}(\lambda,\Lambda,\Omega,A_0)$, is also $H$-compact.
 It would be of great interest to compare this result with the alternative notion of nonlocal $H$-convergence introduced in~\cite{Waurick,Waurick2}, which is based on a functional analytic framework. 

In the second part of this paper, we focus on the symmetric case, and prove Theorem~\ref{thm:Gammacompactness}, which is the $\Gamma$-compactness of the class of nonlocal energies associated with $\mathcal M^{\rm sym}(\lambda,\Lambda,\Omega,A_0)$. 
 
Building on the connection between local and nonlocal settings outlined in~\cite{KrSc22}, the proof of Theorem~\ref{thm:Gammacompactness} follows from the $\Gamma$-compactness of local energies (Proposition~\ref{prop:DalMaso}), and a novel equivalence between the $\Gamma$-convergence of nonlocal energies and that of their corresponding local counterparts (Proposition~\ref{prop:loc-nonloc}).
A central ingredient in the proof of Proposition~\ref{prop:loc-nonloc} is the uniqueness of the integral representation of nonlocal energies in terms of $L^\infty$-coefficients, in complete analogy with the local case, see Lemma~\ref{lem:uniqueness}.

As shown in~\cite{KrSc22} in the context of relaxation problems, and in~\cite{CuKrSc23} for functionals depending on a fractional gradient with finite horizon, the $\Gamma$-compactness of energies depending on fractional gradients is expected to hold for a more general class of functionals, not necessarily of quadratic type. Nevertheless, in this paper we focus on the quadratic case, as our main interest lies in exploring the interplay between $\Gamma$-convergence and $H$-convergence of linear fractional operators.

A very non-exhaustive list of recent references about the $\Gamma$-convergence of fractional quadratic energies is given by~\cite{ABSS23,BraidesDalMaso,Savin,solci}.
In particular, in~\cite{BraidesDalMaso}, the $\Gamma$-compactness is obtained using the Beurling-Deny criteria for Dirichlet forms, which are unfortunately not applicable in our setting, as it is not known whether quadratic forms with the fractional gradient are Dirichlet forms. 

The final part of the paper remains within the symmetric setting.
 In Theorem~\ref{thm:equiv_Gamma-H}, we prove the equivalence between nonlocal $H$-convergence and $\Gamma$-convergence of the associated energies. 
To this end, in Lemma~\ref{lem:equiv_Gamma-G}, we first relate the $\Gamma$-convergence of nonlocal energies to the notion of nonlocal $G$-convergence (Definition~\ref{def:nonlocalG}). 
 We then prove, in Proposition~\ref{prop:convmom}, the convergence of the nonlocal momenta via $\Gamma$-convergence. As a consequence, and in view of Theorem~\ref{thm:Gammacompactness}, we obtain a purely variational alternative proof of Theorem~\ref{thm:sym-compact}. 

In a forthcoming paper, we plan to extend these results to the setting of \textit{monotone operators} with $p$-growth, $p\ge 2$, generalizing the well-known results~\cite{CDMD1990,FMT04,Tartar} to the fractional setting. 
\medskip

\begin{figure}
\begin{center}
\begin{tikzpicture}[node distance=.5cm and 1cm, auto, >=stealth, thick]


 \node (HLoc) [rectangle, draw, rounded corners, fill=gray!20] {Local $H$-convergence};
 \node (HNonloc) [rectangle, draw, rounded corners, fill=gray!20, right=3.5cm of HLoc] {Nonlocal $H$-convergence};
 \node (GammaLoc) [rectangle, draw, rounded corners, fill=gray!20, below=3cm of HLoc] {Local $\Gamma$-convergence};
 \node (GammaNonloc) [rectangle, draw, rounded corners, fill=gray!20, below=3cm of HNonloc] {Nonlocal $\Gamma$-convergence};

 \draw[<->] (GammaNonloc) -- node[right]{Thm~\ref{thm:equiv_Gamma-H}} (HNonloc);
 \draw[<->] (GammaLoc) -- node[above]{} (HLoc);
 \draw[->] (HLoc) -- node[above]{Thm~\ref{thm:sym-compact}} (HNonloc);
 \draw[<->] (GammaLoc) -- node[below]{Prop.~\ref{prop:loc-nonloc}} (GammaNonloc);

 \node (l) at ($(HLoc)!0.5!(GammaNonloc) + (-1.6,-0.2)$) {\scriptsize\textbf{Local}};
 \node (nl1) at ($(HLoc)!0.5!(GammaNonloc) + (1.6,-0.2)$) {\scriptsize\textbf{Nonlocal}};
 \node at ($(l)+(-1,0)$) {\tiny $w \in H^1$};
\node at ($(nl1)+(1.3,0)$) {\tiny $u \in H^s$};
 \node at ($(nl1)!0.5!(l)+(0,1)$) {\tiny $u := (-\Delta)^{\frac{1-s}{2}} w$};
\node at ($(nl1)!0.5!(l)+(0,-1)$) {\tiny $w := I_{1-s} u$};
 \draw[->, bend left=30] (nl1) to (l);
 \draw[->, bend left=30] (l) to (nl1);

\end{tikzpicture}
\end{center}
\caption{The relations between $H$- and $\Gamma$-convergence in the case of symmetric matrices in the local and nonlocal scenarios.}
\end{figure}

The paper is organized as follows. 
In Section~\ref{sec:prel}, we collect some preliminary results.
Section~\ref{sec:H-comp} is devoted to the analysis of nonlocal $H$-convergence. 
 We prove the $H$-compactness result in the general case (Theorem~\ref{thm:compactness}), the uniqueness of the $H$-limit (Lemma~\ref{lem:uniquenessgeneral}), and the $H$-compactness in the symmetric case (Theorem~\ref{thm:sym-compact}). 
In Section~\ref{sec:Gamma-comp} and Section~\ref{sec:H-Gamma-equiv}, we focus on the symmetric setting. 
We prove the $\Gamma$-compactness result (Theorem~\ref{thm:Gammacompactness}), establish the equivalence between $\Gamma$-convergence and $H$-convergence (Theorem~\ref{thm:equiv_Gamma-H}), and provide an alternative variational proof of Theorem~\ref{thm:sym-compact}. 
Finally, in Section~\ref{sec:conclusions}, we outline a number of open problems and future research directions arising from our results.


\section{Preliminaries}\label{sec:prel}


\subsection{Notation}
We assume that $s\in(0,1)$, $n\ge 2$, and that $\Omega$ is a bounded open subset of $\R^n$.
The space of $n\times n$ real matrices is denoted by $\R^{n\times n}$, while $\R^{n\times n}_{\rm sym}$ is the subspace of symmetric matrices.
We adopt standard notation for Lebesgue spaces on measurable subsets $E\subseteq\mathbb R^n$ and Sobolev spaces on open subsets $O\subseteq\mathbb R^n$. We let $\|\cdot\|_X$ denote the norm of a Banach space $X$, $X'$ its dual space, and $\langle \cdot,\cdot\rangle_{X'\times X}$ the duality pairing between $X'$ and $X$.


\subsection{The functional setting}\label{sec:Riesz}
For every $\alpha\in (0,n)$, the $\alpha$-Riesz potential of a measurable function $f\colon\R^n\to\R$ is defined as
\[
I_\alpha f(x)\coloneqq \frac{1}{\gamma_\alpha}\int_{\R^n} \frac{f(y)}{|x-y|^{n-\alpha}}\,{\rm d} y,\qquad\gamma_\alpha\coloneqq 2^\alpha\pi^\frac{n}{2}\frac{\Gamma(\frac{\alpha}{2})}{\Gamma(\frac{n-\alpha}{2})}.
\]
Notice that $I_\alpha$ is a Fourier multiplier that acts in the Fourier space as
\[
\mathcal{F}\left(I_\alpha f\right)(\xi)=|\xi|^{-\alpha}\mathcal{F}\left(f\right)(\xi),
\]
being $\mathcal F$ the Fourier transform. Moreover, as shown in~\cite[Theorem~1, pag.~119]{Stein}, we have the following result.

\begin{prop}\label{prop:Is}
Let $\alpha\in(0,n)$ and $p\in \left(1,\frac{n}{\alpha}\right)$.
For every $f\in L^p(\R^n)$, the $\alpha$-Riesz potential $I_\alpha f$ is well-defined and there exists a positive constant $C$, depending only on $\alpha$, $n$, and $p$, such that
\begin{equation*}
\left\|I_\alpha f\right\|_{L^{p^*_\alpha}(\R^n)}\le C\left\|f\right\|_{L^p(\R^n)},\quad\text{where }p^*_\alpha\coloneqq\frac{np}{n-\alpha p}.
\end{equation*}
\end{prop}
Let us recall the notions of fractional divergence and fractional gradient that will be used to introduce the $H$-convergence problem in our framework.
\begin{defn}\label{def:divnablafraz}
Let $\psi\in C_c^\infty(\R^n)$ and $\phi\in C^\infty_c(\R^n;\R^n)$ be fixed. We define the $s$-fractional gradient $\nabla^s\psi$ and the $s$-fractional divergence $\div^s\phi$, respectively, as
\begin{align*}
\nabla^s\psi(x)&\coloneqq \frac{n-1+s}{\gamma_{1-s}}\int_{\mathbb{R}^n}\frac{(\psi(x)-\psi(y))(x-y)}{|x-y|^{n+s+1}}\,{\rm d}y\quad\text{for all $x\in\R^n$},\\
\div^s\phi(x)&\coloneqq \frac{n-1+s}{\gamma_{1-s}}\int_{\mathbb{R}^n}\frac{(\phi(x)-\phi(y))\cdot(x-y)}{|x-y|^{n+s+1}}\,{\rm d}y\quad\text{for all $x\in\R^n$}.
\end{align*}
\end{defn}
The next result links the notions of fractional gradient and fractional divergence with the classical ones.
A proof can be found in~\cite[Theorem~1.2]{ShiehSpector} and~\cite[Proposition~2.2]{ComiStefani19}.
\begin{prop}\label{prop:identities}
Let $\psi\in C_c^\infty(\R^n)$ and $\phi\in C^\infty_c(\R^n;\R^n)$. It holds that
\begin{align*}
\nabla^s\psi&=\nabla(I_{1-s}\psi)=I_{1-s}(\nabla\psi)\quad\text{in $\R^n$},\\
\div^s\phi&=I_{1-s}(\div\phi)=\div(I_{1-s}\phi)\quad\text{in $\R^n$}.
\end{align*}
\end{prop}
We now introduce the functional framework for our problems in fractional divergence form.
\begin{defn}\label{def:sobolev_spaces}
Let $O\subseteq\R^n$ be an open set. We let $H_0^s(O)$ denote the Hilbert space defined as the closure of $C_c^\infty(O)$ with respect to the following norm
\[
\left\|\,\cdot\,\right\|_{H^s(\R^n)}\coloneqq\left(\left\|\,\cdot\,\right\|^2_{L^2(\R^n)}+\left\|\nabla^s\,\cdot\,\right\|^2_{L^2(\R^n;\R^n)}\right)^{1/2}
\]
and, by $H^{-s}(O)$, the dual space of $H^s_0(O)$. When $O=\R^n$, we simply write $H^s_0(\R^n)=H^s(\R^n)$.
\end{defn}
As a consequence of Proposition~\ref{prop:identities}, one can immediately show, through Fourier transform, that for all $\psi\in C_c^\infty(\R^n)$ and $s,\sigma\in(0,1)$ we have
\begin{equation}\label{eq:composition}
-\div^s(\nabla^{\sigma}\psi)=(-\Delta)^{\frac{s+\sigma}{2}}\psi.
\end{equation}
Moreover, by the Fubini's Theorem, it holds that
\begin{equation}\label{eq:duality}
\int_{\R^n}\nabla^s\psi(x)\cdot\phi(x)\,{\rm d}x=-\int_{\R^n}\psi(x)\div^s\phi(x)\,{\rm d} x
\end{equation}
for all $\psi\in C_c^\infty(\R^n)$ and $\phi\in C^\infty_c(\R^n;\R^n)$.

If we extend the operators $\nabla^s$ and $\div^s$ respectively to $H^s(\R^n)$ and $H^s(\R^n;\R^n)$, then~\eqref{eq:duality} holds for all $\psi\in H^s(\R^n)$ and $\phi\in H^s(\R^n;\R^n)$.
In particular, we can define the operator $\div^s\colon L^2(\R^n;\R^n)\to H^{-s}(\Omega)$ as
\[
\langle \div^s u,v\rangle_{H^{-s}(\Omega)\times H^s_0(\Omega)}\coloneqq -\int_{\R^n}u(x)\cdot \nabla^s v(x)\,{\rm d}x
\]
for all $u\in L^2(\R^n;\R^n)$ and $v\in H^s_0(\Omega)$.

The following two propositions provide a useful connection between the Sobolev spaces $H^1(\R^n)$ and $H^s(\R^n)$. 
For the proof of the next proposition we refer the interested reader to {\cite[Lemma~A.4]{BCCS22}}.
\begin{prop}\label{prop:fr-Lap}
Let $w\in H^1(\R^n)$ and set $u\coloneqq (-\Delta)^{\frac{1-s}{2}}w$.
It holds that 
\begin{align*}
 u\in H^s(\R^n)\quad\text{and}\quad\nabla^s u=\nabla w\text{ a.e.\ in $\R^n$}.
\end{align*}
Moreover, there exists a positive constant $C$, depending only on $n$ and $s$, such that the following estimate holds true
\begin{equation}\label{eq:interpolationBCCS}
\|u\|_{L^2(\R^n)}\le C(n,s)\|w\|_{L^2(\R^n)}^s\|\nabla w\|_{L^2(\R^n;\R^n)}^{1-s}.
\end{equation}
\end{prop}
\begin{prop}\label{prop:Riesz}
Let $u\in H^s(\R^n)$ and set $w\coloneqq I_{1-s}u$.
It holds that 
\[
w\in \left\{v\in L^{2^*_{1-s}}(\R^n)\,:\, \nabla v\in L^2(\R^n)\right\}\quad\text{and}\quad\nabla w=\nabla^s u\text{ a.e.\ in $\R^n$},
\]
where $2^*_{1-s}=\frac{2n}{n-2+2s}$.

In particular, $w\in H^1_{\rm loc}(\R^n)$ and, for every open set $O\Subset\R^n$, there exists a positive constant $C$, depending only on $n$, $s$ and $O$, such that the following estimate holds true
\begin{equation}\label{eq:continuityestriesz}
\|w\|_{H^1(O)}\le C(n,s)\|u\|_{H^s(\R^n)}.
\end{equation}
\end{prop}
\begin{proof}
If $u\in C_c^\infty(\R^n)$, the result is a consequence of Proposition~\ref{prop:Is} (with $p=2$ and $\alpha=1-s$) and of Proposition~\ref{prop:identities}.

Let now $(u_h)_h\subset C_c^\infty(\R^n)$ be such that $u_h\to u$ strongly in $H^s(\R^n)$ as $h\to\infty$. Then,
\[
w_h\coloneqq I_{1-s}u_h\to w\coloneqq I_{1-s}u\quad\text{strongly in }L^{2^*_{1-s}}(\R^n)\text{ as }h\to\infty
\]
and
\[
\nabla w_h=\nabla^s u_h\to \nabla^s u\quad\text{strongly in }L^2(\R^n;\R^n)\text{ as }h\to\infty,
\]
which gives the existence of $\nabla w=\nabla^s u\in L^2(\R^n;\R^n)$.
Thus, the estimate~\eqref{eq:continuityestriesz} follows again in virtue of Proposition~\ref{prop:Is}.
\end{proof}

We conclude this subsection by recalling some extensions of classical results to the fractional setting that will be used throughout the paper.
Namely: a Leibniz-type rule for the fractional gradient~\cite[Eq.~(1.5),~(1.6)]{ComiStefani22} and {\cite[Eq.~(2.11)]{KrSc22}}, a Poincaré-type inequality~\cite[Theorem~3.3]{ShiehSpector}, and a Rellich-type Theorem~\cite[Theorem~2.2]{ShiehSpectorII}. 

\begin{prop}[Leibniz rule]\label{prop:Leibniz}
Let $\varphi\in C_c^1(\Omega)$ and $u\in H^{s}(\R^n)$. Then, $\varphi u\in H^{s}_0(\Omega)$ and
\[
\nabla^s(\varphi u)=\varphi\nabla^s u+u\nabla^s\varphi+\nabla^s_{\rm NL}(\varphi,u)
\]
where, for every $x\in\R^n$, the remainder term $\nabla^s_{\rm NL}(\varphi,u)(x)$ is
\[
\nabla^s_{\rm NL}(\varphi,u)(x)\coloneqq\frac{n-1+s}{\gamma_{1-s}}\int_{\R^n}\frac{(\varphi(x)-\varphi(y))(u(x)-u(y))(x-y)}{|x-y|^{n+s+1}}\,{\rm d}y.
\] 
Moreover, there exists a positive constant $C$, depending only on $n$ and $s$, such that
\[
\|\nabla^s_{\rm NL}(\varphi,u)\|_{L^2(\R^n;\R^n)}\le C\|u\|_{L^2(\R^n)}\|\varphi\|_{L^\infty(\Omega)}^{1-s}\|\nabla \varphi\|_{L^\infty(\Omega;\R^n)}^s.
\]
\end{prop}

\begin{prop}[Poincaré inequality]\label{prop:Poincare}
For every open set $O\Subset\R^n$ there exists a positive constant $C$, depending only on $n$, $s$ and $O$, such that 
\[
\|u\|_{L^2(O)}\le C\|\nabla^s u\|_{L^2(\R^n;\R^n)}\quad\text{for all $u\in H^s(\R^n)$}.
\]
\end{prop}
\begin{prop}[Rellich Theorem]\label{prop:Rellich}
For every open set $O\Subset\R^n$ the space $H^s_0(O)$ is compactly embedded into $L^2(O)$.
\end{prop}


\subsection{The \texorpdfstring{$H$}{H}-convergence problem}\label{sec:H-conv}

We now want to introduce our notion of $H$-convergence associated with the nonlocal operators introduced above.
As in the local counterpart (see e.g.~\cite{Tartar}), we begin by defining the following classes of matrices.
\begin{defn}
Given $0<\lambda\le \Lambda<\infty$ and a measurable subset $E\subseteq\mathbb{R}^n$, we define $\mathcal M(\lambda,\Lambda,E)$ as the collection of all matrix-valued measurable functions $A\colon E\to\R^{n\times n}$ satisfying
\begin{align}
&A(x)\xi\cdot \xi\ge \lambda|\xi|^2& &\text{for all $\xi\in\R^n$ and for a.e.\ $x\in E$},\label{eq:unif-bound1}\\
&A(x)\xi\cdot\xi \ge \Lambda^{-1}|A(x)\xi|^2& &\text{for all $\xi\in\R^n$ and for a.e.\ $x\in E$}.\label{eq:unif-bound2}
\end{align}
We also set
\[
\mathcal M^{\rm sym}(\lambda,\Lambda,E)\coloneqq \{A\in \mathcal M(\lambda,\Lambda,E): \text{$A=A^T$ a.e.\ in $E$}\}.
\]
\end{defn}

The estimate~\eqref{eq:unif-bound2} above is needed to obtain the compactness of the class $\mathcal{M}(\lambda,\Lambda,\Omega)$ with respect to the $H$-convergence topology, even in the local setting.
More precisely, it is well-known that the $H$-limit of sequences of matrices satisfying the standard growth condition
\begin{align}\label{eq:unif-bound3}
\lambda|\xi|^2\leq A(x)\xi\cdot \xi\leq\Lambda|\xi|^2\quad\text{for all $\xi\in\R^n$ and for a.e.\ }x\in\Omega,
\end{align}
instead of~\eqref{eq:unif-bound1} and~\eqref{eq:unif-bound2}, may belong to a wider class $\mathcal M(\lambda,\Lambda',\Omega)$ with $\Lambda'\geq\Lambda$ (see e.g.\ the observations of Tartar in~\cite[Chapter~6, Pag.\ 81]{Tartar}).

We also point out that, in the symmetric case $A=A^T$, conditions~\eqref{eq:unif-bound1}--\eqref{eq:unif-bound2} are actually equivalent to~\eqref{eq:unif-bound3}, as shown in the following result.

\begin{lem}\label{lem:equiv}
Given $0< \lambda\le \Lambda<\infty$, let $B\in \mathbb R^{n\times n}$ satisfy
\begin{align}
B\xi\cdot \xi\ge \lambda|\xi|^2\quad\text{for all $\xi\in\R^n$}.\label{eq:Bunif-bound0}
\end{align}
Then, $B$ is invertible and condition
\begin{align}\label{eq:Bunif-bound1}
B\xi\cdot\xi \ge \Lambda^{-1}|B\xi|^2\quad\text{for all $\xi\in\R^n$}
\end{align}
is equivalent to
\begin{align}
B^{-1}\xi\cdot\xi \ge \Lambda^{-1}|\xi|^2\quad\text{for all $\xi\in\R^n$}.\label{eq:Bunif-bound2}
\end{align}
In particular, $B$ satisfies
\begin{equation}\label{eq:Bunif-bound3}
B\xi\cdot\xi\le \Lambda|\xi|^2\quad\text{for all $\xi\in\R^n$}.
\end{equation}
Moreover, if $B$ is symmetric, then conditions~\eqref{eq:Bunif-bound1},~\eqref{eq:Bunif-bound2} and~\eqref{eq:Bunif-bound3} are all equivalent. 
\end{lem}

\begin{proof}
Let $B\in \R^{n\times n}$ satisfy~\eqref{eq:Bunif-bound0}. By the Lax-Milgram Theorem, for all $\eta\in\R^n$ there exists a unique vector $\xi\in\R^n$ such that $B\xi=\eta$, i.e.\ $B$ is invertible. Hence,~\eqref{eq:Bunif-bound1} implies
\[
\Lambda^{-1}|\xi|^2=\Lambda^{-1}|BB^{-1}\xi|^2\le BB^{-1}\xi\cdot B^{-1}\xi=B^{-1}\xi\cdot\xi\quad\text{for all $\xi\in\R^n$}.
\]
Conversely, by~\eqref{eq:Bunif-bound2}, we get
\[
\Lambda^{-1}|B\xi|^2\le B^{-1}B\xi\cdot B\xi=B\xi\cdot\xi\quad\text{for all $\xi\in\R^n$}.
\]

Assume now that $B$ satisfies condition~\eqref{eq:Bunif-bound1}. By the Cauchy-Schwartz inequality, we get
\[
\Lambda^{-1}|B\xi|^2\le B\xi\cdot \xi\leq |B\xi| |\xi|\quad\text{for all $\xi\in\R^n$},
\]
leading to
\begin{equation}\label{eq:B-Lambda}
|B\xi|\leq\Lambda|\xi|.
\end{equation}
Therefore, by applying again the Cauchy-Schwartz inequality, we get
\[
B\xi\cdot\xi\le |B\xi||\xi|\le \Lambda|\xi|^2\quad\text{for all $\xi\in\R^n$}
\]
and so condition~\eqref{eq:Bunif-bound1} implies~\eqref{eq:Bunif-bound3}. 

Finally, assume that $B$ is symmetric. It is enough to show that~\eqref{eq:Bunif-bound3} implies~\eqref{eq:Bunif-bound1}. We start by observing that, by the symmetry of $B$, the bilinear form $(\xi,\eta)\mapsto (B\xi,\eta)$ is a scalar product in $\R^n$, in virtue of~\eqref{eq:Bunif-bound3}.
Hence, by the Cauchy-Schwartz inequality and~\eqref{eq:Bunif-bound3}, it holds that
\[
(B\xi\cdot \eta)^2\le (B\xi\cdot\xi) (B\eta\cdot\eta)\le (B\xi\cdot \xi)\Lambda |\eta|^2\quad\text{for all $\xi,\eta\in\R^n$}.
\]
In particular, for $\eta=B\xi$, we get
\[
|B\xi|^4=(B\xi\cdot B\xi)^2
\le (B\xi\cdot \xi)\Lambda|B\xi|^2\quad\text{for all $\xi\in\R^n$},
\]
which implies~\eqref{eq:Bunif-bound1}.
\end{proof}

From now on, we fix $0<\lambda\le \Lambda<\infty$. Given a sequence of matrices $(A_h)_h\subset\mathcal M(\lambda,\Lambda,\R^n)$, for all $f\in H^{-s}(\Omega)$ and $h\in\mathbb N$ we consider the following elliptic problems
\begin{equation*}\tag{$P_h^f$}
\begin{cases}
-\div^s(A_h\nabla^s u_h)=f&\text{in }\Omega,\\
u_h=0\quad&\text{in }\R^n\setminus\Omega,
\end{cases}
\end{equation*}
where $\div^s$ and $\nabla^s$ are the fractional-order differential operators introduced in Definition~\ref{def:divnablafraz}.
The following result ensures that problems $(P_h^f)$ are well-posed. 

\begin{lem}\label{lem:LaxMilgram}
Let $A\in\mathcal M(\lambda,\Lambda,\R^n)$. For every $f\in H^{-s}(\Omega)$ there exists a unique weak solution $u\in H^s_0(\Omega)$ of the elliptic problem
\begin{equation*}
\begin{cases}
-\div^s(A\nabla^s u)=f&\text{in }\Omega,\\
u=0\quad&\text{in }\R^n\setminus\Omega,
\end{cases}
\end{equation*}
i.e.\ satisfying
\[
\int_{\R^n} A(x)\nabla^s u(x)\cdot\nabla^s v(x)\,{\rm d}x=\langle f,v\rangle_{H^{-s}(\Omega)\times H^s_0(\Omega)}\quad\text{for all $v\in H^s_0(\Omega)$}.
\]
Moreover, the solution $u$ satisfies the following estimate
\begin{align*}
\|\nabla^s u\|_{L^2(\R^n;\R^n)}\leq\lambda^{-1}\|f\|_{H^{-s}(\Omega)}.
\end{align*}
\end{lem}

\begin{proof}
The proof is a direct application of the Lax-Milgram Theorem and Proposition~\ref{prop:Poincare}.
Indeed, the bilinear form $a:H^s_0(\Omega)\times H^s_0(\Omega)\to\mathbb{R}$, defined as
\[
a(u,v)\coloneqq\int_{\R^n} A(x)\nabla^s u(x)\cdot\nabla^s v(x)\,{\rm d}x\quad\text{for all }u,v\in H^s_0(\Omega),
\]
is continuous and coercive, being $A\in\mathcal M(\lambda,\Lambda,\R^n)$.
\end{proof}

The notion of nonlocal $H$-convergence of $(A_h)_h\subset \mathcal M(\lambda,\Lambda,\R^n)$ we propose consists in finding a limit matrix $A_\infty\in\mathcal M(\lambda',\Lambda',\R^n)$, with $0<\lambda'\le\lambda\le\Lambda\le\Lambda'<\infty$, such that the sequence of problems $(P_h^f)_h$ is related to the limit problem
\begin{equation*}\tag{$P_\infty^f$}
\begin{cases}
-\div^s(A_\infty\nabla^s u_\infty)=f&\text{in }\Omega, \\
u_\infty=0\quad&\text{in }\R^n\setminus\Omega,
\end{cases}
\end{equation*}
in the sense of the next definition.
In what follows, we let $u_h=u_h(f)\in H^s_0(\Omega)$, $h\in\mathbb N$, and $u_\infty=u_\infty(f)\in H^s_0(\Omega)$ denote the unique weak solutions of $(P_h^f)$ and $(P_\infty^f)$, respectively.

\begin{defn}[Nonlocal $H$-convergence]\label{def:Hscon}
Let $0<\lambda'\le \lambda\le \Lambda\le \Lambda'<\infty$ and consider $(A_h)_h\subset\mathcal M(\lambda,\Lambda,\R^n)$ and $A_\infty \in \mathcal M(\lambda',\Lambda',\R^n)$.
We say that 
\[
\text{$(A_h)_h$ $H$-converges to $A_\infty$ in $H^s_0(\Omega)$}
\]
if for all $f\in H^{-s}(\Omega)$ the following convergences simultaneously hold as $h\to\infty$:
\begin{align}
 &\textbf{convergence of solutions:}\quad u_h\to u_\infty\quad\text{weakly in }H^s_0(\Omega),\label{eq:Hscon1}\\
 &\textbf{convergence of momenta:}\quad A_h\nabla^s u_h\to A_\infty\nabla^s u_\infty
\quad\text{weakly in }L^2(\R^n;\R^n).\label{eq:Hscon2}
\end{align}
\end{defn}

Definition~\ref{def:Hscon} is the natural counterpart in the nonlocal setting of the local $H$-convergence, see e.g.~\cite[Definition 6.4]{Tartar}. For the readers' convenience, we recall the notion of local $H$-convergence in what follows. 

Let $(B_h)_h\subset \mathcal M(\lambda,\Lambda,\Omega)$ be fixed. For every $h\in\mathbb N$ and $g\in H^{-1}(\Omega)$, we consider the following sequence of elliptic problems
\begin{equation*}\tag{$Q_h^g$}
\begin{cases}
-\div(B_h\nabla w_h)=g&\text{in }\Omega, \\
w_h=0\quad&\text{on }\partial\Omega.
\end{cases}
\end{equation*}
Given $B_\infty \in \mathcal M(\lambda',\Lambda',\Omega)$, for some $0<\lambda'\le\lambda\le\Lambda\le\Lambda'<\infty$, we also consider the problem
\begin{equation*}\tag{$Q_\infty^g$}
\begin{cases}
-\div(B_\infty\nabla w_\infty)=g&\text{in }\Omega, \\
w_\infty=0\quad&\text{on }\partial\Omega,
\end{cases}
\end{equation*}
and we let $w_h=w_h(g)\in H^1_0(\Omega)$, $h\in\mathbb N$, and $w_\infty=w_\infty(g)\in H^1_0(\Omega)$ denote the unique weak solutions of $(Q_h^g)$ and $(Q_\infty^g)$, respectively.

\begin{defn}[Local $H$-convergence]
Let $0<\lambda'\le \lambda\le \Lambda\le \Lambda'<\infty$, and consider $(B_h)_h\subset\mathcal M(\lambda,\Lambda,\Omega)$ and $B_\infty \in \mathcal M(\lambda',\Lambda',\Omega)$.
We say that 
\[
\text{$(B_h)_h$ $H$-converges to $B_\infty$ in $H^1_0(\Omega)$},
\]
if for all $g\in H^{-1}(\Omega)$ the following convergences simultaneously hold as $h\to\infty$:
\begin{align}
 &\textbf{convergence of solutions:}\quad w_h\to w_\infty\quad\text{weakly in }H^1_0(\Omega),\\
 &\textbf{convergence of momenta:}\quad B_h\nabla w_h\to B_\infty\nabla w_\infty
\quad\text{weakly in }L^2(\Omega;\R^n).
\end{align}
\end{defn}

We recall that the class $\mathcal M(\lambda,\Lambda,\Omega)$ is compact with respect to the local $H$-convergence in $H^1_0(\Omega)$. A proof of this classical result can be found e.g.\ in~\cite[Theorem~6.5]{Tartar}.

\begin{prop}[Local $H$-compactness]\label{prop:Tartar}
For every $(B_h)_h\subset \mathcal M(\lambda,\Lambda,\Omega)$, there exist a not relabeled subsequence and a matrix-valued function $B_\infty\in\mathcal M(\lambda,\Lambda,\Omega)$ such that 
\[
\text{$(B_h)_h$ $H$-converges to $B_\infty$ in $H^1_0(\Omega)$}.
\]
\end{prop}

We also recall that the local $H$-convergence is stable with respect to the transpose operation. More precisely, we have the following result, due to Tartar, whose proof can be found in {\cite[Lemma~10.2]{Tartar}}.

\begin{prop}[$H$-convergence of the transpose]\label{prop:transpose}
Let $B_h,B_\infty\in\mathcal M(\lambda,\Lambda,\Omega)$, $h\in\mathbb N$, and assume that
\[
\text{$(B_h)_h$ $H$-converges to $B_\infty$ in $H^1_0(\Omega)$}.
\]
Then,
\[
\text{$(B_h^T)_h$ $H$-converges to $B_\infty^T$ in $H^1_0(\Omega)$}.
\]
\end{prop}

The first goal of this paper is to show that suitable subclasses of $\mathcal M(\lambda,\Lambda,\R^n)$ are compact with respect to the \textit{nonlocal} $H$-convergence. We introduce the following classes of matrices.

\begin{defn}
For every $A_0\in\mathcal M(\lambda,\Lambda,\R^n)$, we define
\[
\mathcal M(\lambda,\Lambda,\Omega,A_0)\coloneqq\left\{A\in\mathcal M(\lambda,\Lambda,\R^n):A=A_0\text{ a.e.\ in $\R^n\setminus\Omega$}\right\}.
\]
For every $A_0\in\mathcal M^{\rm sym}(\lambda,\Lambda,\R^n)$, we also set
\[
\mathcal M^{\rm sym}(\lambda,\Lambda,\Omega,A_0)\coloneqq \{A\in \mathcal M(\lambda,\Lambda,\Omega,A_0): \text{$A=A^T$ a.e.\ in $\R^n$}\}.
\]
\end{defn}

In the next section we show that any sequence $(A_h)_h$ in $\mathcal M(\lambda,\Lambda,\Omega,A_0)$ (respectively in $\mathcal M^{\rm sym}(\lambda,\Lambda,\Omega,A_0)$), for a fixed matrix $A_0\in\mathcal M(\lambda,\Lambda,\R^n)$ (respectively in $\mathcal M^{\rm sym}(\lambda,\Lambda,\R^n)$), admits a not relabeled subsequence and a limit matrix $A_\infty$ in $\mathcal M(\lambda,\Lambda,\Omega,A_0)$ (respectively in $\mathcal M^{\rm sym}(\lambda,\Lambda,\Omega,A_0)$) such that
\[
\text{$(A_h)_h$ $H$-converges to $A_\infty$ in $H^s_0(\Omega)$}
\]
in the sense of Definition~\ref{def:Hscon}. In particular, the $H$-limit $A_\infty$ satisfies~\eqref{eq:unif-bound1}--\eqref{eq:unif-bound2} with the same constants $\lambda$ and $\Lambda$ of the sequence $(A_h)_h$, leading to the compactness of both classes.


\subsection{The \texorpdfstring{$\Gamma$}{Gamma}-convergence problem}

Let us fix $A_0\in\mathcal M^{\rm sym}(\lambda,\Lambda,\R^n)$ and a sequence $(A_h)_h$ in $\mathcal M^{\rm sym}(\lambda,\Lambda,\Omega,A_0)$.
For every $h\in\mathbb N$, we consider the nonlocal energies $F_h\colon L^2(\R^n)\to [0,\infty]$ associated with $A_h$, represented by 
\begin{align}\label{eq:Fh}
F_h(u)\coloneqq\begin{cases}
\displaystyle\frac{1}{2}\int_{\R^n}A_h(x)\nabla^s u(x)\cdot\nabla^s u(x)\,{\rm d}x&\text{if $u\in H^s_0(\Omega)$},\\
\displaystyle\infty&\text{if $u\in L^2(\R^n)\setminus H^s_0(\Omega)$}.
\end{cases}
\end{align}
The second goal of this paper is to show the $\Gamma$-compactness in the strong topology of $L^2(\R^n)$ for the class of nonlocal energies~\eqref{eq:Fh}.
In particular, we prove the existence of a limit symmetric matrix-valued function $A_\infty \in \mathcal M^{\rm sym}(\lambda,\Lambda,\Omega,A_0)$ such that, up to a not relabeled subsequence,
\[
\text{$(F_h)_h$ $\Gamma$-converges to $F_\infty$ strongly in $L^2(\R^n)$},
\]
where the limit nonlocal energy $F_\infty\colon L^2(\R^n)\to [0,\infty]$ is represented by
\begin{equation}\label{eq:Finfty}
F_\infty(u)\coloneqq
\begin{cases}
\displaystyle\frac{1}{2}\int_{\R^n}A_\infty(x)\nabla^s u(x)\cdot\nabla^s u(x)\,{\rm d}x&\text{if $u\in H^s_0(\Omega)$},\\
\displaystyle\infty&\text{if $u\in L^2(\R^n)\setminus H^s_0(\Omega)$}.
\end{cases}
\end{equation}

We start recalling the notion of $\Gamma$-convergence for functionals.
For a complete treatment of the topic, we refer the interested reader to the monographs~\cite{Bra,DalMaso}.

\begin{defn}\label{def:Gamma}
Let $X$ be a Banach space and let $E_h,E_\infty\colon X\to [0,\infty]$, $h\in\mathbb N$. We say that
\[
\text{$(E_h)_h$ $\Gamma$-converges to $E_\infty$ strongly in $X$}
\]
if the following two conditions simultaneously hold:
\begin{itemize}
\item {\bf $\Gamma$-liminf inequality:} for every $x\in X$ and for every sequence $(x_h)_h\subset X$ strongly converging to $x$ in $X$ one has
\begin{equation*}
E_\infty(x)\le\liminf_{h\to\infty} E_h(x_h);
\end{equation*}
\item {\bf $\Gamma$-limsup inequality:} for every $x\in X$ there exists a recovery sequence $(y_h)_h\subset X$ strongly converging to $x$ in $X$ and such that 
\begin{equation*}
E_\infty(x)\ge\limsup_{h\to\infty} E_h(y_h).
\end{equation*}
\end{itemize}
\end{defn}
In analogy with Subsection~\ref{sec:H-conv}, we remind that the class of local energies $G_h\colon L^2(\Omega)\rightarrow [0,\infty]$, $h\in\mathbb N$, associated with a sequence $(B_h)_h$ in $\mathcal M^{\rm sym}(\lambda,\Lambda,\Omega)$ and represented by
\begin{align}\label{eq:Gh}
G_h(v)\coloneqq
\begin{cases}
\displaystyle\frac{1}{2}\int_{\Omega}B_h(x)\nabla v(x)\cdot\nabla v(x)\,{\rm d}x&\text{if $v\in H^1(\Omega)$},\\
\displaystyle\infty&\text{if $v\in L^2(\Omega)\setminus H^1(\Omega)$},
\end{cases}
\end{align}
is compact with respect to the $\Gamma$-convergence in the strong topology of $L^2(\Omega)$.
A proof of the next result can be found in~\cite{Sbordone} and~\cite[Theorem~22.2]{DalMaso}.

\begin{prop}[$\Gamma$-compactness of local functionals]\label{prop:DalMaso}
Let $(B_h)_h\subset \mathcal M^{\rm sym}(\lambda,\Lambda,\Omega)$ and, for every $h\in\mathbb N$, let $G_h\colon L^2(\Omega)\rightarrow [0,\infty]$ be the local energy associated with $B_h$ as in~\eqref{eq:Gh}. 
Then, there exists $G_\infty\colon L^2(\Omega)\to [0,\infty]$ 
such that, up to a not relabeled subsequence,
\[
\text{$(G_h)_h$ $\Gamma$-converges to $G_\infty$ strongly in $L^2(\Omega)$}. 
\]
Moreover, there exists a matrix-valued function $B_\infty\in\mathcal M^{\rm sym}(\lambda,\Lambda,\Omega)$ such that the $\Gamma$-limit $G_\infty$ has the following integral representation 
\begin{align}\label{eq:Ginfty}
G_\infty(v)\coloneqq\begin{cases}
\displaystyle\frac{1}{2}\int_{\Omega}B_\infty(x)\nabla v(x)\cdot\nabla v(x)\,{\rm d}x&\text{if $v\in H^1(\Omega)$},\\
\displaystyle\infty&\text{if $v\in L^2(\Omega)\setminus H^1(\Omega)$}.
\end{cases}
\end{align}
\end{prop}
\begin{rmk}
In order to derive Proposition~\ref{prop:DalMaso} from~\cite{Sbordone} and~\cite[Theorem~22.2]{DalMaso}, we recall that, according to Lemma~\ref{lem:equiv}, any symmetric matrix-valued function $B$ belongs to the space $\mathcal M^{\rm sym}(\lambda,\Lambda,\Omega)$ if and only if the following growth condition is satisfied
\begin{align}\label{eq:classicbounds}
\lambda|\xi|^2\leq B(x)\xi\cdot \xi\leq\Lambda|\xi|^2\quad\text{for all $\xi\in\R^n$ and for a.e.\ }x\in\Omega.
\end{align}
\end{rmk}

\subsection{The equivalence problem}

The third goal of this paper is to show the equivalence between nonlocal $H$-convergence and $\Gamma$-convergence of the associated nonlocal energies. 

Let us fix $A_0\in\mathcal M^{\rm sym}(\lambda,\Lambda,\R^n)$. For every $h\in\mathbb N$, let $A_h,A_\infty\in\mathcal M^{\rm sym}(\lambda,\Lambda,\Omega,A_0)$ and let $F_h,F_\infty\colon L^2(\R^n)\to [0,\infty]$ be the corresponding nonlocal energies, respectively defined in~\eqref{eq:Fh} and~\eqref{eq:Finfty}. We want to show that 
\[
\text{$(A_h)_h$ $H$-converges to $A_\infty$ in $H^s_0(\Omega)$}
\]
if and only if
\[
\text{$(F_h)_h$ $\Gamma$-converges to $F_\infty$ strongly in $L^2(\R^n)$}.
\]
In order to prove the above equivalence, we need to introduce the following notion of {\it nonlocal $G$-convergence} of $(A_h)_h\subset \mathcal M^{\rm sym}(\lambda,\Lambda,\R^n)$.

\begin{defn}[Nonlocal $G$-convergence]\label{def:nonlocalG}
Let $0<\lambda'\le \lambda\le\Lambda\le\Lambda'<\infty$ and consider $(A_h)_h\subset \mathcal M^{\rm sym}(\lambda,\Lambda,\R^n)$ and $A_\infty\in \mathcal M^{\rm sym}(\lambda',\Lambda',\R^n)$. We say that 
\[
\text{$(A_h)_h$ $G$-converges to $A_\infty$ in $H^s_0(\Omega)$}
\]
if for all $f\in H^{-s}(\Omega)$ the following convergence holds as $h\to\infty$:
\begin{equation}\label{eq:Gscon}
 u_h\to u_\infty\quad\text{weakly in }H^s_0(\Omega),
\end{equation}
where $u_h=u_h(f)$ and $u_\infty=u_\infty(f)$ denote, respectively, the unique solutions of $(P_h^f)$ and $(P_\infty^f)$.
\end{defn}

Definition~\ref{def:nonlocalG} is the natural counterpart in the nonlocal setting of the classical definition of $G$-convergence, introduced by Spagnolo in~\cite{Spagnolo}. 
Clearly, nonlocal $H$-convergence implies nonlocal $G$-convergence, see Definition~\ref{def:Hscon}. 

In Section~\ref{sec:H-Gamma-equiv}, we show that the nonlocal $G$-convergence of $(A_h)_h$ is equivalent to the $\Gamma$-convergence of the associated nonlocal energies $(F_h)_h$. 
Moreover, we prove that for every sequence $(A_h)_h\subset \mathcal M^{\rm sym}(\lambda,\Lambda,\Omega,A_0)$, the $\Gamma$-convergence of the associated nonlocal energies $(F_h)_h$ implies the convergence of the nonlocal momenta. 

As a consequence, in the class $\mathcal M^{\rm sym}(\lambda,\Lambda,\Omega,A_0)$ the notions of nonlocal $H$-convergence, nonlocal $G$-convergence, and $\Gamma$-convergence of the nonlocal energies, are all equivalent. 
 

\section{\texorpdfstring{$H$}{H}-compactness of nonlocal operators}\label{sec:H-comp}

This first section is dedicated to the following $H$-compactness result for nonlocal operators. The proof of the following Theorem~\ref{thm:compactness} relies on Proposition~\ref{prop:Tartar} and Proposition~\ref{prop:transpose}. 

\begin{theorem}\label{thm:compactness}
Let $A_0\in \mathcal M(\lambda,\Lambda,\R^n)$.
For every $(A_h)_h\subset\mathcal M(\lambda,\Lambda,\Omega,A_0)$, there exist a not relabeled subsequence and a matrix-valued function $A_\infty\in\mathcal M(\lambda,\Lambda,\Omega,A_0)$ such that
\[
\text{$(A_h)_h$ $H$-converges to $A_\infty$ in $H^s_0(\Omega)$}.
\]
\end{theorem}
\begin{proof}
For every $h\in\mathbb N$, we define
\begin{equation}\label{eq:localMatrices}
 B_h\coloneqq A_h|_\Omega\in\mathcal M(\lambda,\Lambda,\Omega).
\end{equation}
By Proposition~\ref{prop:Tartar}, there exists a limit matrix
\begin{equation}\label{eq:existenceHlimitLocal}
B_\infty\in\mathcal M(\lambda,\Lambda,\Omega) 
\end{equation}
such that, up to subsequences,
\[
(B_h)_h\text{ $H$-converges to }B_\infty\text{ in }H^1_0(\Omega). 
\]
For a.e.\ $x\in\R^n$ we define 
\begin{align}\label{eq:Ainfty}
 A_\infty(x)\coloneqq
 \begin{cases}
 B_\infty(x)&\text{if $x\in\Omega$},\\
 A_0(x)&\text{if $x\in\R^n\setminus\Omega$}.
\end{cases}
\end{align}
By construction, $A_\infty \in \mathcal M(\lambda,\Lambda,\Omega,A_0)$.
To conclude the proof, we show that $A_\infty$ is the $H$-limit of $(A_h)_h$ in $H^s_0(\Omega)$, in the sense of Definition~\ref{def:Hscon}.

Fix $f\in H^{-s}(\Omega)$ and consider the sequence $(u_h)_h\subset H^s_0(\Omega)$ of unique solutions of $(P_h^f)_h$. 
Since $(u_h)_h$ and $(A_h\nabla^s u_h)_h$ are respectively bounded in $H^s_0(\Omega)$ and $L^2(\R^n;\R^n)$, then there exist $u^*\in H^s_0(\Omega)$ and $m\in L^2(\R^n;\R^n)$ such that, up to a not relabeled subsequence,
\begin{equation}\label{eq:Hweak}
u_h\to u^*\quad\text{weakly in $H^s_0(\Omega)$}\quad\text{and}\quad A_h\nabla^su_h\to m\quad\text{weakly in $L^2(\R^n;\R^n)$ as $h\to\infty$}.
\end{equation}
Our goal is to show that 
\begin{equation}
 \label{eq:keyidentity}
 m=A_\infty\nabla^s u^*\quad\text{a.e.\ in $\R^n$}.
\end{equation}
Indeed, once obtained~\eqref{eq:keyidentity}, by passing to the limit (as $h\to\infty$) in the weak formulation of $(P_h^f)$, we obtain that $u^*$ is a solution of $(P_\infty^f)$.
Thus, by uniqueness, $u^*=u_\infty$ and, by~\eqref{eq:Hweak},
\[
\text{$(A_h)_h$ $H$-converges to $A_\infty$ in $H^s_0(\Omega)$,}
\]
as desired.

First, notice that
\begin{equation}\label{eq:keyidentity1}
m=A_0\nabla^s u^*=A_\infty\nabla^s u^*\quad\text{a.e.\ in $\R^n\setminus\Omega$}.
\end{equation}
In fact, by~\eqref{eq:Hweak}, for all $\Phi\in L^2(\R^n\setminus\Omega;\R^n)$ it holds that
\begin{align*}
\int_{\R^n\setminus\Omega}m(x)\cdot\Phi(x)\,{\rm d}x&=\lim_{h\to\infty}\int_{\R^n\setminus\Omega}A_h(x)\nabla^s u_h(x)\cdot\Phi(x)\,{\rm d}x\\
&=\lim_{h\to\infty}\int_{\R^n\setminus\Omega}A_0(x)\nabla^s u_h(x)\cdot\Phi(x)\,{\rm d}x=\int_{\R^n\setminus\Omega}A_0(x)\nabla^s u^*(x)\cdot\Phi(x)\,{\rm d}x,
\end{align*}
which implies~\eqref{eq:keyidentity1}.

It remains to show that 
\begin{equation}\label{eq:keyidentity2}
m=A_\infty\nabla^s u^*\quad\text{a.e.\ in $\Omega$}.
\end{equation}
Fix $g\in H^{-1}(\Omega)$ and consider $(w_h)_h\subset H^1_0(\Omega)$ and $w_\infty\in H^1_0(\Omega)$, respectively unique solutions of the transpose problems
\begin{equation}\label{eq:def-wh}
\begin{cases}
-\div(B_h^T\nabla w_h)=g&\text{in }\Omega, \\
w_h=0\quad&\text{on }\partial\Omega,
\end{cases}\quad\text{and}\quad
\begin{cases}
-\div(B_\infty^T\nabla w_\infty)=g&\text{in }\Omega, \\
w_\infty=0\quad&\text{on }\partial\Omega.
\end{cases}
\end{equation}
By Proposition~\ref{prop:transpose},
\begin{equation}\label{eq:wh-winfty}
w_h\to w_\infty\text{ weakly in $H^1_0(\Omega)$ and }B_h^T\nabla w_h\to B_\infty^T\nabla w_\infty\text{ weakly in $L^2(\Omega;\R^n)$ as $h\to\infty$}.
\end{equation}

Let $\varphi\in C_c^\infty(\Omega)$ and define
\begin{equation*}
M_h\coloneqq \int_{\R^n}\varphi(x)A_h(x)\nabla^s u_h(x)\cdot \nabla w_h(x)\,{\rm d}x.
\end{equation*}
We claim that for all $\varphi\in C_c^\infty(\Omega)$
\begin{equation}\label{eq:m-identity}
\lim_{h\to\infty}M_h=\int_\Omega \varphi(x) m(x)\cdot\nabla w_\infty(x)\,{\rm d}x=\int_\Omega \varphi(x) A_\infty(x)\nabla^s u^*(x)\cdot\nabla w_\infty(x)\,{\rm d}x.
\end{equation} 
Once the claim is proved, we obtain~\eqref{eq:keyidentity2}. Indeed, since the operator $-\div(B^T_\infty\nabla\,\cdot\,)$ defines a bijection between the spaces $H^1_0(\Omega)$ and $H^{-1}(\Omega)$ and, since $g\in H^{-1}(\Omega)$ can be arbitrarily taken in~\eqref{eq:def-wh}, then for all $\varphi,\psi\in C_c^\infty(\Omega)$ the following identity holds
\begin{equation*}
\int_\Omega \varphi(x) m(x)\cdot \nabla \psi(x) \,{\rm d}x=\int_\Omega \varphi(x) A_\infty(x)\nabla^s u^*(x)\cdot \nabla \psi(x) \,{\rm d}x.
\end{equation*}
Hence, 
\begin{equation}\label{eq:m-identity2}
m(x)\cdot\nabla \psi(x)=A_\infty(x)\nabla^s u^*(x)\cdot\nabla \psi(x)\quad\text{for a.e.\ $x\in\Omega$ and for all $\psi\in C_c^\infty(\Omega)$}, 
\end{equation}
and the collections of points of $\Omega$ where~\eqref{eq:m-identity2} fails can be chosen independent of $\psi$.

Let us fix $\Omega'\Subset\Omega$ and $\phi\in C_c^\infty(\Omega)$, such that $\phi=1$ on $\Omega'$, and define
\[
\psi(x) \coloneqq \phi(x)\xi\cdot x\quad\text{for a.e.\ $x\in\Omega$ and for all $\xi\in\R^n$}.
\]
By~\eqref{eq:m-identity2}, we get that
\begin{equation*}
m(x)\cdot\xi=A_\infty(x)\nabla^s u^*(x)\cdot\xi\quad\text{for a.e.\ $x\in\Omega'$ and for all $\xi\in\R^n$},
\end{equation*}
which implies the validity of~\eqref{eq:keyidentity2} in $\Omega'$.
Moreover, since this is true for every $\Omega'\Subset \Omega$, we get~\eqref{eq:keyidentity2} in all $\Omega$, which completes the proof.

We then conclude by showing the validity of the claim~\eqref{eq:m-identity}.
Its proof is divided into two steps.

\smallskip

\textbf{Step 1.} We first show that
\begin{equation}\label{eq:Mh-1}
\lim_{h\to\infty}M_h=\int_\Omega \varphi(x)A_\infty(x)\nabla^s u^*(x)\cdot\nabla w_\infty(x)\,{\rm d}x.
\end{equation}
In virtue of Proposition~\ref{prop:identities}, we have
\begin{align}
M_h&=\int_\Omega\varphi(x)\nabla^s u_h(x)\cdot B_h^T(x)\nabla w_h(x)\,{\rm d}x\nonumber\\
&=\int_\Omega\varphi(x)\nabla (I_{1-s}u_h)(x)\cdot B_h^T(x)\nabla w_h(x)\,{\rm d}x\nonumber\\
&=\int_\Omega\nabla (\varphi I_{1-s}u_h)(x)\cdot B_h^T(x)\nabla w_h(x)\,{\rm d}x-\int_\Omega I_{1-s}u_h(x)\nabla\varphi(x)\cdot B_h^T(x)\nabla w_h(x)\,{\rm d}x.\label{eq:Mh-1a}
\end{align}
By~\eqref{eq:Hweak} and Proposition~\ref{prop:Rellich},
\begin{equation}\label{eq:uh-rellich}
u_h\to u^*\quad\text{strongly in $L^2(\R^n)$ as $h\to\infty$}, 
\end{equation}
and so, by Proposition~\ref{prop:Is}, we get
\[
I_{1-s}u_h\to I_{1-s}u^*\quad\text{strongly in $L^2(\Omega)$ as $h\to\infty$}.
\]
This last convergence, coupled with~\eqref{eq:wh-winfty}, implies that
\begin{align}\label{eq:Mh-1b}
\lim_{h\to\infty}\int_\Omega I_{1-s}u_h(x)\nabla\varphi(x)\cdot B_h^T(x)\nabla w_h(x)\,{\rm d}x&=\int_\Omega I_{1-s}u^*(x)\nabla\varphi(x)\cdot B_\infty^T(x)\nabla w_\infty(x)\,{\rm d}x.
\end{align}
We observe that $\varphi I_{1-s}u_h\in H^1_0(\Omega)$. Hence, by~\eqref{eq:Hweak},~\eqref{eq:uh-rellich}, and Proposition~\ref{prop:Riesz},
\[
\varphi I_{1-s}u_h\to \varphi I_{1-s}u^*\quad\text{weakly in $H^1_0(\Omega)$ as $h\to\infty$}.
\]
Thus, since $w_h$ solves problem~\eqref{eq:def-wh}, we get
\begin{align}
\lim_{h\to\infty}\int_\Omega\nabla (\varphi I_{1-s}u_h)(x)\cdot B_h^T(x)\nabla w_h(x)\,{\rm d}x&=\lim_{h\to\infty}\langle g,\varphi I_{1-s}u_h\rangle_{H^{-1}(\Omega)\times H^1_0(\Omega)}\nonumber\\
&=\langle g,\varphi I_{1-s}u^*\rangle_{H^{-1}(\Omega)\times H^1_0(\Omega)}\nonumber\\
&=\int_\Omega\nabla (\varphi I_{1-s}u^*)(x)\cdot B_\infty^T(x)\nabla w_\infty(x)\,{\rm d}x.\label{eq:Mh-1c}
\end{align}
By combining~\eqref{eq:Mh-1a},~\eqref{eq:Mh-1b}, and~\eqref{eq:Mh-1c}, we then obtain~\eqref{eq:Mh-1}, being
\begin{align*}
\lim_{h\to\infty}M_h&=\int_\Omega\nabla (\varphi I_{1-s}u^*)(x)\cdot B_\infty^T(x)\nabla w_\infty(x)\,{\rm d}x-\int_\Omega I_{1-s}u^*(x)\nabla\varphi(x)\cdot B_\infty^T(x)\nabla w_\infty(x)\,{\rm d}x\\
&=\int_\Omega \varphi(x)\nabla(I_{1-s}u^*)(x)\cdot B_\infty^T(x)\nabla w_\infty(x)\,{\rm d}x=\int_\Omega \varphi(x)A_\infty(x)\nabla^s u^*(x)\cdot\nabla w_\infty(x)\,{\rm d}x.
\end{align*}

\smallskip

\textbf{Step 2.} We conclude by showing that
\begin{equation}\label{eq:Mh-2}
\lim_{h\to\infty}M_h=\int_\Omega \varphi(x) m(x)\cdot\nabla w_\infty(x)\,{\rm d}x.
\end{equation}
By Proposition~\ref{prop:fr-Lap} and Proposition~\ref{prop:Leibniz}, we have
\begin{align}
M_h&=\int_\Omega\varphi(x)A_h(x)\nabla^s u_h(x)\cdot\nabla w_h(x)\,{\rm d}x\nonumber\\
&=\int_{\R^n}\varphi(x)A_h(x)\nabla^s u_h(x)\cdot\nabla^s((-\Delta)^\frac{1-s}{2}w_h)(x)\,{\rm d}x\nonumber\\
&=\int_{\R^n} A_h(x)\nabla^s u_h(x)\cdot\nabla^s(\varphi(-\Delta)^\frac{1-s}{2}w_h)(x)\,{\rm d}x\label{eq:Mh-2a}\\
&\quad-\int_{\R^n} A_h(x)\nabla^s u_h(x)\cdot\nabla^s\varphi(x) (-\Delta)^{\frac{1-s}{2}}w_h(x)\,{\rm d}x\nonumber\\
&\quad-\int_{\R^n} A_h(x)\nabla^s u_h(x)\cdot\nabla^s_{\rm NL}(\varphi,(-\Delta)^\frac{1-s}{2}w_h)(x)\,{\rm d}x.\nonumber
\end{align}
For what concerns the second integral in~\eqref{eq:Mh-2a}, in view of~\eqref{eq:wh-winfty} and Proposition~\ref{prop:Rellich}, we have
\[
w_h\to w_\infty\quad\text{strongly in $L^2(\R^n)$ as $h\to\infty$}.
\]
Moreover, the sequence $(w_h)_h\subset H^1_0(\Omega)$ is uniformly bounded. Hence, by Proposition~\ref{prop:fr-Lap},
\begin{equation}\label{eq:wh-strong1}
(-\Delta)^\frac{1-s}{2}w_h\to (-\Delta)^\frac{1-s}{2}w_\infty\quad\text{strongly in $L^2(\R^n)$}
\end{equation}
and, since $\nabla^s\varphi\in L^\infty(\R^n)$, by~\eqref{eq:Hweak} we get
\begin{align}
&\lim_{h\to\infty}\int_{\R^n} A_h(x)\nabla^s u_h(x)\cdot\nabla^s\varphi(x) (-\Delta)^{\frac{1-s}{2}}w_h(x)\,{\rm d}x
=\int_{\R^n} m(x)\cdot\nabla^s\varphi(x) (-\Delta)^{\frac{1-s}{2}}w_\infty(x)\,{\rm d}x.\label{eq:Mh-2b}
\end{align}
Regarding the third integral in~\eqref{eq:Mh-2a}, by Proposition~\ref{prop:Leibniz} and~\eqref{eq:wh-strong1}, since $\nabla^s_{\rm NL}$ is a bilinear operator, we deduce that
\[
\nabla^s_{\rm NL}(\varphi,(-\Delta)^\frac{1-s}{2}w_h)\to \nabla^s_{\rm NL}(\varphi,(-\Delta)^\frac{1-s}{2}w_\infty)\quad\text{strongly in $L^2(\R^n;\R^n)$ as $h\to\infty$}.
\]
Therefore, in virtue of~\eqref{eq:Hweak}, we have
\begin{align}
 &\lim_{h\to\infty}\int_{\R^n} A_h(x)\nabla^s u_h(x)\cdot\nabla^s_{\rm NL}(\varphi,(-\Delta)^\frac{1-s}{2}w_h)(x)\,{\rm d}x\nonumber \\&
 =\int_{\R^n} m(x)\cdot\nabla^s_{\rm NL}(\varphi,(-\Delta)^\frac{1-s}{2}w_\infty)(x)\,{\rm d}x.\label{eq:Mh-2c}
\end{align}
Finally, $(-\Delta)^\frac{1-s}{2}w_h\in H^s(\R^n)$ by Proposition~\ref{prop:fr-Lap}, which implies that $\varphi (-\Delta)^\frac{1-s}{2}w_h\in H^s_0(\Omega)$.
Then,
\[
\varphi (-\Delta)^\frac{1-s}{2}w_h\rightarrow \varphi (-\Delta)^\frac{1-s}{2}w_\infty\quad\text{weakly in $H^s_0(\Omega)$ as $h\to\infty$}.
\]
Therefore, regarding the first integrals in~\eqref{eq:Mh-2a}, since $u_h$ is a solution of $(P_h^f)$ then, by~\eqref{eq:Hweak},
\begin{align}
&\lim_{h\to\infty}\int_{\R^n} A_h(x)\nabla^s u_h(x)\cdot\nabla^s(\varphi(-\Delta)^\frac{1-s}{2}w_h)(x)\,{\rm d}x\nonumber\\
&=\lim_{h\to\infty}\langle f,\varphi(-\Delta)^\frac{1-s}{2}w_h\rangle_{H^{-s}(\Omega)\times H^s_0(\Omega)}\nonumber\\
&=\langle f,\varphi(-\Delta)^\frac{1-s}{2}w_\infty\rangle_{H^{-s}(\Omega)\times H^s_0(\Omega)}\nonumber\\
&=\lim_{h\to\infty}\int_{\R^n} A_h(x)\nabla^s u_h(x)\cdot\nabla^s(\varphi(-\Delta)^\frac{1-s}{2}w_\infty)(x)\,{\rm d}x\nonumber\\
&=\int_{\R^n} m(x)\cdot\nabla^s(\varphi(-\Delta)^\frac{1-s}{2}w_\infty)(x)\,{\rm d}x.\label{eq:Mh-2d}
\end{align}
By combining~\eqref{eq:Mh-2a},~\eqref{eq:Mh-2b},~\eqref{eq:Mh-2c}, and~\eqref{eq:Mh-2d}, we finally get~\eqref{eq:Mh-2}, being
\begin{align*}
\lim_{h\to\infty}M_h&=
\int_{\R^n} m(x)\cdot\nabla^s(\varphi(-\Delta)^\frac{1-s}{2}w_\infty)(x)\,{\rm d}x\\
&\quad- \int_{\R^n} m(x)\cdot\nabla^s\varphi(x) (-\Delta)^{\frac{1-s}{2}}w_\infty(x)\,{\rm d}x-\int_{\R^n} m(x)\cdot\nabla^s_{\rm NL}(\varphi,(-\Delta)^\frac{1-s}{2}w_\infty)(x)\,{\rm d}x\\
&= \int_{\R^n} \varphi(x) m(x)\cdot\nabla^s((-\Delta)^\frac{1-s}{2}w_\infty)(x)\,{\rm d}x=\int_\Omega \varphi(x) m(x)\cdot\nabla w_\infty(x)\,{\rm d}x.
\end{align*}
Hence, the claim~\eqref{eq:m-identity} holds true and the proof of the theorem is accomplished.
\end{proof}
 
\begin{rmk}
The main application of the $H$-convergence of local linear elliptic operators lies in the periodic homogenisation of operators of the form
\begin{equation}
\label{eq:localhomo}
-\div(a(hx)\nabla u(x)),\quad h\in\mathbb N,
\end{equation}
where $a\in L^\infty(\R^n)$ is $1$-periodic, and satisfies $a\ge\lambda>0$ a.e.\ in $\R^n$. 
We expect that, in the fractional counterpart of~\eqref{eq:localhomo}, the assumption of fixing the function $a$ outside the reference domain $\Omega$, which is required in Theorem~\ref{thm:compactness}, can be dropped.
\end{rmk}
 
We point out that, as in the local scenario, in the class $\mathcal M(\lambda,\Lambda,\Omega,A_0)$ the $H$-limit $A_\infty$ is unique.

\begin{lem}\label{lem:uniquenessgeneral}
Let $A_0\in \mathcal M(\lambda,\Lambda,\R^n)$. Assume that $(A_h)_h\subset\mathcal M(\lambda,\Lambda,\Omega,A_0)$ and $A_\infty,\widehat A_\infty \in \mathcal M(\lambda,\Lambda,\Omega,A_0)$ satisfy
\begin{equation}\label{eq:H-A}
\text{$(A_h)_h$ $H$-converges to $A_\infty$ in $H^s_0(\Omega)$}
\end{equation}
and
\begin{equation}\label{eq:H-hatA}
\text{$(A_h)_h$ $H$-converges to $\widehat A_\infty$ in $H^s_0(\Omega)$}.
\end{equation}
Then,
\begin{equation}\label{eq:H-uniq}
A_\infty(x)=\widehat A_\infty(x)\quad\text{for a.e.\ $x\in\R^n$}.
\end{equation}
\end{lem}

\begin{proof}
By definition
\[
A_\infty(x)=A_0(x)=\widehat A_\infty(x)\quad\text{for a.e.\ $x\in\R^n\setminus\Omega$}.
\]
Hence, it remains to show that 
\[
A_\infty(x)=\widehat A_\infty(x)\quad\text{for a.e.\ $x\in\Omega$}.
\]

Let $\{{\rm e}_1,\dots, {\rm e}_n\}$ be the canonical orthonormal basis in $\R^n$. By~\cite[Lemma~4.3]{KrSc22}, for all $r>0$ and for all $x_0\in\R^n$ we can find $n$ functions $\varphi_1,\dots\varphi_n\in C_c^\infty(B_r(x_0))$ such that 
\begin{equation}\label{eq:varphi_i}
\nabla^s\varphi_i(x_0)={\rm e}_i\quad\text{for all $i\in\{1,\dots,n\}$}.
\end{equation}
Let $f\in H^{-s}(\Omega)$ be fixed and let $u_\infty=u_\infty(f)\in H^s_0(\Omega)$ and $\widehat u_\infty=\widehat u_\infty(f)\in H^s_0(\Omega)$ be, respectively, the unique solutions of 
\begin{equation*}
\begin{cases}
-\div^s(A_\infty\nabla^s u_\infty)=f&\text{in }\Omega, \\
u_\infty=0\quad&\text{in }\R^n\setminus\Omega,
\end{cases}\qquad \begin{cases}
-\div^s(\widehat A_\infty\nabla^s \widehat u_\infty)=f&\text{in }\Omega, \\
\widehat u_\infty=0\quad&\text{in }\R^n\setminus\Omega.
\end{cases}
\end{equation*}
By~\eqref{eq:Hscon1} and~\eqref{eq:H-A}--\eqref{eq:H-hatA}, we derive that $u_\infty=\widehat u_\infty$ and, by~\eqref{eq:Hscon2}, for all $\Phi\in L^2(\R^n;\R^n)$
\[
\int_{\R^n}A_\infty(x)\nabla u_\infty(x)\cdot\Phi(x)\,{\rm d} x=\int_{\R^n}\widehat A_\infty(x)\nabla u_\infty(x)\cdot \Phi(x)\,{\rm d} x.
\]
Since the operator $-\div^s(A_\infty\nabla^s\,\cdot\,)=-\div^s(\widehat A_\infty\nabla^s\,\cdot\,)$ defines a bijection between the spaces $H^s_0(\Omega)$ and $H^{-s}(\Omega)$ and, since $f\in H^{-s}(\Omega)$ can be arbitrarily chosen, we derive that for all $\psi\in C_c^\infty(\Omega)$ and for all $\Phi\in L^2(\R^n;\R^n)$
\[
\int_{\R^n}A_\infty(x)\nabla^s \psi(x)\cdot\Phi(x)\,{\rm d} x=\int_{\R^n}\widehat A_\infty(x)\nabla^s \psi(x)\cdot \Phi(x)\,{\rm d} x.
\]
Hence, 
\begin{equation}\label{eq:H-identity}
A_\infty(x)\nabla^s \psi(x)=\widehat A_\infty(x)\nabla^s \psi(x)\quad\text{for a.e.\ $x\in\Omega$ and for all $\psi\in C_c^\infty(\Omega)$},
\end{equation}
and the collections of points of $\Omega$ where~\eqref{eq:H-identity} fails can be chosen independent of $\psi$. 

Let now $x_0\in \Omega$ and $r>0$ be such that $B_r(x_0)\subset \Omega$. By choosing $\psi=\varphi_i$ for all $i\in\{1,\dots,n\}$, where $\varphi_1,\dots,\varphi_n\in C^\infty_c(B_r(x_0))$ are the functions satisfying~\eqref{eq:varphi_i}, we derive that 
\[
A_\infty(x_0){\rm e}_i=\widehat A_\infty(x_0){\rm e}_i\quad\text{for a.e.\ $x_0\in\Omega$ and for all $i\in\{1,\dots,n\}$}.
\]
Hence $A_\infty(x)=\widehat A_\infty(x)$ for a.e.\ $x\in\Omega$, which gives~\eqref{eq:H-uniq}.
\end{proof}
 
As a consequence of Proposition~\ref{prop:transpose}, the $H$-compactness Theorem~\ref{thm:compactness} also applies to the subclass of symmetric matrices $\mathcal M^{\rm sym}(\lambda,\Lambda,\Omega,A_0)$, for a fixed $A_0\in\mathcal M^{\rm sym}(\lambda,\Lambda,\R^n)$.
More precisely, the $H$-limit $A_\infty$ of sequences $(A_h)_h$ in $\mathcal M^{\rm sym}(\lambda,\Lambda,\Omega,A_0)$ still lies in $\mathcal M^{\rm sym}(\lambda,\Lambda,\Omega,A_0)$.

\begin{theorem}\label{thm:sym-compact}
Let $A_0\in\mathcal M^{\rm sym}(\lambda,\Lambda,\R^n)$.
For every $(A_h)_h\subset\mathcal M^{\rm sym}(\lambda,\Lambda,\Omega,A_0)$, there exist a not relabeled subsequence and a matrix-valued function $A_\infty\in\mathcal M^{\rm sym}(\lambda,\Lambda,\Omega,A_0)$ such that
\[
\text{$(A_h)_h$ $H$-converges to $A_\infty$ in $H^s_0(\Omega)$}.
\]
\end{theorem}

\begin{proof}
The proof of the theorem in the symmetric case follows verbatim the construction already presented in the proof of Theorem~\ref{thm:compactness} for the general case, combined with the following observations.
The limit matrix $B_\infty$ in the local $H$-convergence of the sequence $(B_h)_h$, introduced in~\eqref{eq:localMatrices}, whose existence is shown in~\eqref{eq:existenceHlimitLocal}, is now symmetric in view of Proposition~\ref{prop:transpose} and by the uniqueness of the $H$-limit. 
This implies that the matrix $A_\infty$, defined in~\eqref{eq:Ainfty}, belongs to the class $\mathcal M^{\rm sym}(\lambda,\Lambda,\Omega,A_0)$, and it is the $H$-limit of $(A_h)_h$, in virtue of Theorem~\ref{thm:compactness}. 
\end{proof}


\section{\texorpdfstring{$\Gamma$}{Gamma}-compactness of nonlocal energies}\label{sec:Gamma-comp}

The goal of this section is to show the following $\Gamma$-compactness theorem.

\begin{theorem}\label{thm:Gammacompactness}
Let $A_0\in\mathcal M^{\rm sym}(\lambda,\Lambda,\R^n)$. Let $(A_h)_h\subset\mathcal M^{\rm sym}(\lambda,\Lambda,\Omega,A_0)$ and let $(F_h)_h$ be the nonlocal energies introduced in~\eqref{eq:Fh}.
Then, there exist a not relabeled subsequence of $(A_h)_h$ and $A_\infty\in\mathcal M^{\rm sym}(\lambda,\Lambda,\Omega,A_0)$ such that 
\[
\text{$(F_h)_h$ $\Gamma$-converges to $F_\infty$ strongly in $L^2(\R^n)$}, 
\]
where $F_\infty\colon L^2(\R^n)\to [0,\infty]$ is the nonlocal energy associated with $A_\infty$, as in~\eqref{eq:Finfty}.
\end{theorem}
 
In order to prove Theorem~\ref{thm:Gammacompactness}, we begin with the following lemma, which guarantees the uniqueness of the integral representation of nonlocal energies by matrices.

\begin{lem}\label{lem:uniqueness}
Let $A,\widehat A\in L^\infty(\R^n;\R^{n\times n}_{\rm sym})$ and let $\Omega\subseteq\R^n$ be an open set. Assume that 
\[
A(x)\xi\cdot\xi\ge 0\quad\text{and}\quad\widehat A(x)\xi\cdot\xi\ge 0\quad\text{for a.e.\ $x\in\Omega$ and all $\xi\in\R^n$},
\]
and
\begin{equation}\label{eq:main-condition}
\int_{\R^n}A(x)\nabla^s\psi(x)\cdot\nabla^s\psi(x)\,{\rm d}x=\int_{\R^n}\widehat A(x)\nabla^s\psi(x)\cdot\nabla^s\psi(x)\,{\rm d}x\quad\text{for all $\psi\in C_c^\infty(\Omega)$}.
\end{equation}
Then,
\begin{equation*}
A(x)=\widehat A(x)\quad\text{for a.e.\ $x\in\Omega$}.
\end{equation*}
\end{lem}

\begin{proof}
We define
\[
C\coloneqq A-\widehat A\in L^\infty(\R^n;\R^{n\times n}_{\rm sym}).
\]
Since $C\in L^\infty(\R^n;\R^{n\times n}_{\rm sym})$, for a.e.\ $x_0\in\R^n$ and for all $M\in (0,\infty)$ we have
\[
\lim_{r\to \infty}\int_{B_M(0)}\left|C\left(\frac{y}{r}+x_0\right)-C(x_0)\right|\,{\rm d}y=0.
\]

In particular, there exists $(r_k)_k\subset (0,\infty)$ with $r_k\to\infty$ as $k\to\infty$ such that 
\begin{equation}\label{eq:Lebesgue}
\left|C\left(\frac{y}{r_k}+x_0\right)-C(x_0)\right|\to 0 \quad\text{as $k\to\infty$ for a.e.\ $y\in \R^n$}.
\end{equation}

{\bf Step 1: Blow-up}. We fix $\varphi\in C_c^\infty(\R^n)$. For all $x_0\in\Omega$ and $r\in (0,\infty)$ we define
\[
\varphi_{x_0,r}(x)\coloneqq \varphi(r(x-x_0))\quad\text{for all $x\in\R^n$}.
\]
There exists $r_0=r_0(x_0,\Omega)\in (0,\infty)$ such that
\[
\varphi_{x_0,r}\in C_c^\infty(\Omega)\quad\text{for all $r\in (r_0,\infty)$}.
\]
We have that
\begin{align*}
\nabla^s\varphi_{x_0,r}(x)=r^s\nabla^s\varphi(r(x-x_0)).
\end{align*}
By~\eqref{eq:main-condition}, for all $x_0\in\Omega$ and $r\in (r_0,\infty)$ we have
\begin{align}
0&=r^{n-2s}\int_{\R^n}C(x)\nabla^s\varphi_{x_0,r}(x)\cdot\nabla^s\varphi_{x_0,r}(x)\,{\rm d}x\nonumber\\
&=r^n\int_{\R^n}C(x)\nabla^s\varphi(r(x-x_0))\cdot\nabla^s\varphi(r(x-x_0))\,{\rm d}x\nonumber\\
&=\int_{\R^n}C\left(\frac{y}{r}+x_0\right)\nabla^s\varphi(y)\cdot\nabla^s\varphi(y)\,{\rm d}y.\label{eq:blow-up}
\end{align}
Moreover, by~\eqref{eq:Lebesgue}, we derive
\[
C\left(\frac{y}{r_k}+x_0\right)\nabla^s\varphi(y)\cdot\nabla^s\varphi(y)\to C(x_0)\nabla^s\varphi(y)\cdot \nabla^s\varphi(y)\quad\text{for a.e.\ $y\in\R^n$ as $k\to\infty$}.
\]
Since
\[
\left|C\left(\frac{y}{r_k}+x_0\right)\nabla^s\varphi(y)\cdot\nabla^s\varphi(y)\right|\le \|C\|_{L^\infty\left(\R^n;\R^{n\times n}_{\rm sym}\right)}|\nabla^s\varphi(y)|^2\quad\text{for a.e.\ $y\in\R^n$ and all $k\in\mathbb N$},
\]
in virtue of the Dominated Convergence Theorem and~\eqref{eq:blow-up}, we conclude that 
\[
\int_{\R^n}C(x_0)\nabla^s\varphi(y)\cdot\nabla^s\varphi(y)\,{\rm d}y=0\quad\text{for all $\varphi\in C_c^\infty(\R^n)$}.
\]

{\bf Step 2: Reduction to the local case}. By a density argument, we derive that 
\[
\int_{\R^n}C(x_0)\nabla^s u(y)\cdot\nabla^s u(y)\,{\rm d}y=0\quad\text{for all $u\in H^{s}(\R^n)$}.
\]
Let $v\in H^1(\R^n)$. By Proposition~\ref{prop:fr-Lap}, we have that $u\coloneqq(-\Delta)^\frac{1-s}{2}v\in H^{s}(\R^n)$ and 
\[
\nabla^su(x)=\nabla v(x)\quad\text{for a.e.\ $x\in\R^n$}.
\]
Therefore, we derive that
\[
\int_{\R^n}A(x_0)\nabla v(y)\cdot\nabla v(y)\,{\rm d}y=\int_{\R^n}\widehat A(x_0)\nabla v(y)\cdot\nabla v(y)\,{\rm d}y\quad\text{for all $v\in H^1(\R^n)$}.
\]
The claim follows by~\cite[Lemma~22.5]{DalMaso}.
\end{proof}

As a consequence of Lemma~\ref{lem:uniqueness}, we obtain the following equivalence between $\Gamma$-convergence of the nonlocal energies $(F_h)_h$ and $\Gamma$-convergence of the local ones $(G_h)_h$. The proof is inspired by some recent ideas presented in~\cite{CuKrSc23,KrSc22}.

\begin{prop}\label{prop:loc-nonloc}
Let $A_0\in\mathcal M^{\rm sym}(\lambda,\Lambda,\R^n)$. For every $h\in\mathbb N$, let $A_h,A_\infty \in \mathcal M^{\rm sym}(\lambda,\Lambda,\Omega,A_0)$ and $F_h,F_\infty\colon L^2(\R^n)\to [0,\infty]$ be the nonlocal energies, respectively defined in~\eqref{eq:Fh} and~\eqref{eq:Finfty}. For every $h\in\mathbb N$, define 
\[
B_h\coloneqq A_h|_\Omega\in\mathcal M^{\rm sym}(\lambda,\Lambda,\Omega),\qquad B_\infty\coloneqq A_\infty|_\Omega\in\mathcal M^{\rm sym}(\lambda,\Lambda,\Omega),
\]
and consider the local energies $G_h,G_\infty\colon L^2(\Omega)\rightarrow [0,\infty]$, respectively defined in~\eqref{eq:Gh} and~\eqref{eq:Ginfty}. Then,
\[
\text{$(F_h)_h$ $\Gamma$-converges to $F_\infty$ strongly in $L^2(\R^n)$}
\]
if and only if
\[
\text{$(G_h)_h$ $\Gamma$-converges to $G_\infty$ strongly in $L^2(\Omega)$}.
\]
\end{prop}

\begin{proof}
{\bf Step 1: $\Gamma$-convergence of $(G_h)_h$ implies $\Gamma$-convergence of $(F_h)_h$}. We assume that \begin{equation}\label{eq:GammaLocal}
\text{$(G_h)_h$ $\Gamma$-converges to $G_\infty$ strongly in $L^2(\Omega)$},
\end{equation}
and we want to show that
\[
\text{$(F_h)_h$ $\Gamma$-converges to $F_\infty$ strongly in $L^2(\R^n)$}. 
\]

{\bf $\Gamma$-liminf inequality}. Let $u_h,u\in L^2(\R^n)$, $h\in\mathbb{N}$, be such that $(u_h)_h$ strongly converges to $u$ in $L^2(\R^n)$ as $h\to\infty$. 
We show that 
\[
F_\infty(u)\le\liminf_{h\to\infty}F_h(u_h).
\]
Without loss of generality, we assume that 
\[
\liminf_{h\to\infty}F_h(u_h)<\infty,
\]
the conclusion being otherwise trivial, and that the limit is actually achieved up to a not relabeled subsequence, i.e.
\[
\liminf_{h\to\infty}F_h(u_h)=\lim_{h\to\infty}F_h(u_h).
\]
According to its own definition, $(F_h)_h$ is finite only on $H^s_0(\Omega)$, thus forcing the sequence $(u_h)_h$ to lie therein. Since $(A_h)_h\subset \mathcal M^{\rm sym}(\lambda,\Lambda,\Omega,A_0)$, there exists a positive constant $C$ such that 
\[
\sup_{h\in\mathbb N}\left\|\nabla^s u_h\right\|_{L^2(\R^n;\R^n)}\le C,
\]
which yields that $(u_h)_h$ is uniformly bounded in $H^s_0(\Omega)$.
Then, the limit $u$ also lies on $H^s_0(\Omega)$ and
\begin{align}\label{eq:weak-u}
u_h\to u\quad\text{weakly in $H^s_0(\Omega)$ as $h\to\infty$}.
\end{align}
For every $h\in\mathbb N$, we define
\begin{align*}
v_h\coloneqq I_{1-s}u_h\quad\text{and}\quad v\coloneqq I_{1-s}u.
\end{align*}
By Proposition~\ref{prop:Riesz}, $v_h,v\in H^1(\Omega)$ for every $h\in\mathbb N$ and, by~\eqref{eq:weak-u} and the continuity of the linear operator $I_{1-s}\colon H^s(\R^n)\to H^1(\Omega)$,
\[
v_h\to v\quad\text{strongly in $L^2(\Omega)$ and weakly in $H^1(\Omega)$ as $h\to\infty$}.
\]
By~\eqref{eq:GammaLocal},
\begin{align}\label{eq:Gh-Gammaliminf}
G_\infty(v)\le\liminf_{h\to\infty}G_h(v_h),
\end{align}
in virtue of the $\Gamma$-liminf inequality.
We also note that, by Proposition~\ref{prop:Riesz},
\[
\nabla v_h=\nabla^s u_h
\quad\text{and}\quad\nabla v=\nabla^s u\quad\text{a.e.\ in $\R^n$}.
\]
Thus, we can rephrase~\eqref{eq:Gh-Gammaliminf} as
\[
\frac{1}{2}\int_\Omega B_\infty(x)\nabla^s u(x)\cdot\nabla^s u(x)\,{\rm d}x\le \liminf_{h\to\infty}\frac{1}{2}\int_\Omega A_h(x)\nabla^s u_h(x)\cdot\nabla^s u_h(x)\,{\rm d}x.
\]
On the other hand, since $A_0\in\mathcal M(\lambda,\Lambda,\R^n)$, by~\eqref{eq:weak-u} we get
\[
\frac{1}{2}\int_{\R^n\setminus\Omega} A_0(x)\nabla^s u(x)\cdot\nabla^s u(x)\,{\rm d}x\le \liminf_{h\to\infty}\frac{1}{2}\int_{\R^n\setminus \Omega} A_0(x)\nabla^s u_h(x)\cdot\nabla^s u_h(x)\,{\rm d}x.
\]
Hence,
\begin{align*}
F_\infty(u)&=\frac{1}{2}\int_\Omega B_\infty(x)\nabla^s u(x)\cdot\nabla^s u(x)\,{\rm d}x+\frac{1}{2}\int_{\R^n\setminus\Omega} A_0(x)\nabla^s u(x)\cdot\nabla^s u(x)\,{\rm d}x\\ &\le\liminf_{h\to\infty}\frac{1}{2}\int_\Omega A_h(x)\nabla^s u_h(x)\cdot\nabla^s u_h(x)\,{\rm d}x+\liminf_{h\to\infty}\frac{1}{2}\int_{\R^n\setminus \Omega} A_0(x)\nabla^s u_h(x)\cdot\nabla^s u_h(x)\,{\rm d}x\\
&\le\liminf_{h\to\infty}F_h(u_h).
\end{align*}

{\bf $\Gamma$-limsup inequality}. We fix $u\in L^2(\R^n)$ and show the existence of a recovery sequence $(u_h)_h\subset L^2(\R^n)$ such that $(u_h)_h$ strongly converges to $u$ in $L^2(\R^n)$, as $h\to\infty$, and
\begin{equation}\label{eq:Fh-Gammalimsup}
F_\infty(u)\ge\limsup_{h\to\infty}F_h(u_h).
\end{equation}
The proof of the $\Gamma$-limsup inequality is rather technical and for the readers' convenience, we summarize the main steps below.
\begin{itemize}
 \item First, we exploit the Riesz potential to move to the local setting and we obtain the existence of a recovery sequence $(v_h)_h$ for the $\Gamma$-convergence of the local energies $(G_h)_h$ to $G_\infty$.
 \item Then, through a cut-off argument, we adapt the sequence $(v_h)_h$ to the boundary data of our problem and we come back to the nonlocal setting, obtaining the existence of a sequence $(u_h^\varepsilon)_h$ satisfying the $\Gamma$-limsup inequality up to a reminder term, which depends on a positive parameter $\varepsilon$.
 \item In the last part of the proof, we let the reminder term tend to zero via a diagonal argument, which ensures the existence of a recovery sequence $(u_h)_h$ for our problem.
\end{itemize}
Without loss of generality, we consider only the case of $u\in H^s_0(\Omega)$, the conclusion being otherwise trivial.

We define
\[
v\coloneqq I_{1-s}u.
\]
Then, by Proposition~\ref{prop:Riesz},
\begin{equation}\label{eq:nablav}
v\in H^1(\Omega)\quad\text{and}\quad\nabla v=\nabla^s u\quad\text{a.e.\ in $\R^n$}.
\end{equation}
Moreover, by~\eqref{eq:GammaLocal}, there exists a recovery sequence $(v_h)_h\subset H^1(\Omega)$ for $v$, i.e.\ such that
\begin{equation}\label{eq:new-recovery}
v_h\to v\text{ strongly in $L^2(\Omega)$ as $h\to\infty$}\quad\text{and}\quad\lim_{h\to\infty}G_h(v_h)=G_\infty(v)<\infty.
\end{equation}
We recall, in fact, that the $\Gamma$-liminf and the $\Gamma$-limsup inequalities imply that the limit is achieved at least for the recovery sequence.
In particular, by the definition of $G_h$ (see~\eqref{eq:Gh}), $(v_h)_h$ is bounded in $H^1(\Omega)$, which gives that
\[
v_h\to v\quad\text{weakly in $H^1(\Omega)$ as $h\to\infty$}.
\]

Let $\varepsilon>0$ be fixed and let $K^\varepsilon\Subset \Omega$ be a compact set such that
\begin{equation}\label{eq:fund-est}
\int_{\Omega\setminus K^\varepsilon}|\nabla v(x)|^2\,{\rm d} x<\varepsilon.
\end{equation}
We fix an open set $U^\varepsilon$ such that $K^\varepsilon \Subset U^\varepsilon\Subset\Omega $, consider a cut-off function $\varphi^\varepsilon\in C_c^\infty(U^\varepsilon)$ satisfying $0\le \varphi^\varepsilon\le 1$ on $U^\varepsilon$ and $\varphi^\varepsilon\equiv 1$ on $K^\varepsilon$ and, for every $h\in\mathbb N$, we define
\begin{equation}\label{eq:vhe}
v_h^\varepsilon\coloneqq\varphi^\varepsilon v_h+(1-\varphi^\varepsilon)v.
\end{equation}
By construction,
\begin{equation}\label{eq:vh-v}
v_h^\varepsilon\to v\quad\text{strongly in $L^2(\Omega)$ and weakly in $H^1(\Omega)$ as $h\to\infty$}.
\end{equation}
Moreover, by~\eqref{eq:B-Lambda},~\eqref{eq:classicbounds},~\eqref{eq:fund-est}, and the convexity of the map $\xi\mapsto A_h(x)\xi\cdot\xi$, it holds that
\begin{align*}
G_h(v_h^\varepsilon)&=\frac{1}{2}\int_\Omega A_h(x)[\varphi^\varepsilon (x)\nabla v_h(x)+(1-\varphi^\varepsilon(x))\nabla v(x)]\cdot [\varphi^\varepsilon (x)\nabla v_h(x)+(1-\varphi^\varepsilon(x))\nabla v(x)]\,{\rm d}x\\
&\quad+\int_\Omega A_h(x)[\nabla \varphi^\varepsilon(x) (v_h(x)-v(x))]\cdot [\varphi^\varepsilon (x)\nabla v_h(x)+(1-\varphi^\varepsilon(x))\nabla v(x)]\,{\rm d}x\\
&\quad+\frac{1}{2}\int_\Omega A_h(x)[\nabla \varphi^\varepsilon(x) (v_h(x)-v(x))]\cdot [\nabla \varphi^\varepsilon(x) (v_h(x)-v(x))]\,{\rm d}x\\
&\le\frac{1}{2}\int_\Omega \varphi^\varepsilon(x) A_h(x)\nabla v_h(x)\cdot \nabla v_h(x)\,{\rm d}x+\frac{1}{2}\int_\Omega (1-\varphi^\varepsilon(x)) A_h(x)\nabla v(x)\cdot \nabla v(x)\,{\rm d}x\\
&\quad+\Lambda\|\nabla\varphi^\varepsilon\|_{L^\infty(\Omega;\R^n)}\|v_h-v\|_{L^2(\Omega)}\left(\|\nabla v_h\|_{L^2(\Omega;\R^n)}+\|\nabla v\|_{L^2(\Omega;\R^n)}\right)\\
&\quad+\frac{\Lambda}{2}\|\nabla\varphi^\varepsilon\|_{L^\infty(\Omega;\R^n)}^2\|v_h-v\|_{L^2(\Omega)}^2\\
&\le G_h(v_h)+\frac{\Lambda}{2}\varepsilon+\Lambda\|\nabla\varphi^\varepsilon\|_{L^\infty(\Omega;\R^n)}\|v_h-v\|_{L^2(\Omega)}\left(\|\nabla v_h\|_{L^2(\Omega;\R^n)}+\|\nabla v\|_{L^2(\Omega;\R^n)}\right)\\
&\quad+\frac{\Lambda}{2}\|\nabla\varphi^\varepsilon\|_{L^\infty(\Omega;\R^n)}^2\|v_h-v\|_{L^2(\Omega)}^2.
\end{align*}
Hence, by~\eqref{eq:new-recovery} and the boundedness of $(v_h)_h$ in $H^1(\Omega)$, we conclude that
\begin{align}\label{eq:Glimsup}
&\limsup_{h\to\infty}G_h(v_h^\varepsilon)\le \lim_{h\to\infty}G_h(v_h)+\frac{\Lambda}{2}\varepsilon= G_\infty(v)+\frac{\Lambda}{2}\varepsilon.
\end{align}

We trivially extend $v_h^\varepsilon-v\in H^1_0(\Omega)$ to a function in $H^1(\R^n)$ and, for every $h\in\mathbb N$, we define
\[
w_h^\varepsilon\coloneqq(-\Delta)^{\frac{1-s}{2}}(v_h^\varepsilon-v).
\]
By Proposition~\ref{prop:fr-Lap}, we have that
\begin{equation}\label{eq:nablaswh}
w_h^\varepsilon \in H^s(\R^n)\quad\text{and}\quad\nabla^s w_h^\varepsilon=\nabla (v_h^\varepsilon-v)\quad\text{a.e.\ in $\R^n$} 
\end{equation}
and, by~\eqref{eq:interpolationBCCS},~\eqref{eq:vh-v}, and~\eqref{eq:nablaswh}, there exist two positive constants $C$ and $C_\varepsilon$ such that
\[
\|w_h^\varepsilon\|_{H^s(\R^n)}^2=\|w_h^\varepsilon\|_{L^2(\R^n)}^2+\|\nabla^s w_h^\varepsilon\|_{L^2(\R^n;\R^n)}^2\le C\|v_h^\varepsilon-v\|_{H^1_0(\Omega)}^2\le C_\varepsilon\quad\text{for all $h\in\mathbb N$}.
\]
Therefore, by~\eqref{eq:composition},~\eqref{eq:duality} and~\eqref{eq:vh-v}, for all $\psi\in C_c^\infty(\R^n)$ we get that
\[
\int_{\R^n}w_h^\varepsilon(x)\psi(x)\,{\rm d}x=\int_{\R^n}(v_h^\varepsilon(x)-v(x))(-\Delta)^{\frac{1-s}{2}}\psi(x)\,{\rm d}x\to 0\quad\text{as $h\to\infty$},
\]
which yields that
\begin{align}\label{eq:wh-weak}
w_h^\varepsilon\to 0\quad\text{weakly in $H^s(\R^n)$ as $h\to\infty$}.
\end{align}
In particular, by~\eqref{eq:interpolationBCCS} and~\eqref{eq:vh-v},
\begin{align}\label{eq:wh-strong}
w_h^\varepsilon\to 0\quad\text{strongly in $L^2(\R^n)$ as $h\to\infty$}.
\end{align}

Let $\chi^\varepsilon\in C_c^\infty(\Omega)$ satisfy $0\le \chi^\varepsilon\le 1$ on $\Omega$ and $\chi^\varepsilon=1$ on $\overline{U^\varepsilon}$.
We define
\[
u_h^\varepsilon\coloneqq u+\chi^\varepsilon w_h^\varepsilon\in H^s_0(\Omega).
\]
By~\eqref{eq:wh-weak} and~\eqref{eq:wh-strong},
\begin{equation}\label{eq:uhepsilon1}
u_h^\varepsilon\to u\quad\text{strongly in $L^2(\R^n)$ and weakly in $H^s_0(\Omega)$ as $h\to\infty$}.
\end{equation}
For every $h\in\mathbb N$, we also set
\[
R_h^\varepsilon\coloneqq\nabla^s(\chi^\varepsilon w_h^\varepsilon)-\chi^\varepsilon\nabla^sw_h^\varepsilon.
\]
By Proposition~\ref{prop:Leibniz}, there exists a positive constant $C$ such that
\begin{align*}
\left\|R_h^\varepsilon\right\|_{L^2(\R^n;\R^n)}\le C\left\|\chi^\varepsilon\right\|_{W^{1,\infty}(\R^n)}\left\|w_h^\varepsilon\right\|_{L^2(\R^n)}\quad\text{for all $h\in\mathbb N$}.
\end{align*}
Then, by~\eqref{eq:wh-strong},
\begin{align}\label{eq:Rh-conv}
R_h^\varepsilon\to 0\quad\text{strongly in $L^2(\R^n;\R^n)$ as $h\to\infty$}
\end{align}
and, by~\eqref{eq:nablaswh},
\begin{align}\label{eq:nablauh-for}
\nabla^s u_h^\varepsilon&=\nabla^s u+\chi^\varepsilon\nabla^sw_h^\varepsilon+R_h^\varepsilon=\nabla^s u+\chi^\varepsilon\nabla(v_h^\varepsilon-v)+R_h^\varepsilon\quad\text{a.e.\ in $\R^n$}.
\end{align}
We consider the following decomposition
\begin{align}\label{eq:Fhuh}
F_h(u_h^\varepsilon)&=\frac{1}{2}\int_{U^\varepsilon}A_h(x)\nabla^s u_h^\varepsilon(x)\cdot\nabla^s u_h^\varepsilon(x)\,{\rm d}x+\frac{1}{2}\int_{\Omega\setminus U^\varepsilon}A_h(x)\nabla^s u_h^\varepsilon(x)\cdot\nabla^s u_h^\varepsilon(x)\,{\rm d}x\\
&\quad+\frac{1}{2}\int_{\R^n\setminus\Omega}A_0(x)\nabla^s u_h^\varepsilon(x)\cdot\nabla^s u_h^\varepsilon(x)\,{\rm d}x\nonumber.
\end{align}
By~\eqref{eq:Rh-conv} and~\eqref{eq:nablauh-for}, the last integral in~\eqref{eq:Fhuh} satisfies
\begin{align}\label{eq:Fh-1}
&\lim_{h\to\infty}\frac{1}{2}\int_{\R^n\setminus\Omega}A_0(x)\nabla^s u_h^\varepsilon(x)\cdot\nabla^s u_h^\varepsilon(x)\,{\rm d}x\nonumber\\
 &=\lim_{h\to\infty}\frac{1}{2}\int_{\R^n\setminus\Omega}A_0(x)(\nabla^s u(x)+R_h^\varepsilon(x))\cdot(\nabla^s u(x) +R_h^\varepsilon(x))\,{\rm d}x\nonumber\\
 &=\frac{1}{2}\int_{\R^n\setminus\Omega}A_0(x)\nabla^s u(x)\cdot\nabla^s u(x)\,{\rm d}x.
\end{align}
Concerning the second integral in~\eqref{eq:Fhuh}, we note that, by~\eqref{eq:vhe} and~\eqref{eq:nablauh-for},
\[
\nabla^s u_h^\varepsilon=\nabla^s u+R_h^\varepsilon\quad\text{a.e.\ in }\Omega\setminus U^\varepsilon.
\]
Then, by~\eqref{eq:classicbounds},~\eqref{eq:nablav}, and~\eqref{eq:fund-est},
\begin{align}\label{eq:Fh-2}
 &\limsup_{h\to \infty}\frac{1}{2}\int_{\Omega\setminus U^\varepsilon}A_h(x)\nabla^s u_h^\varepsilon(x)\cdot\nabla^s u_h^\varepsilon(x)\,{\rm d}x\nonumber\\
 &\le\lim_{h\to \infty}\frac{\Lambda}{2}\left\|\nabla^s u+R_h^\varepsilon\right\|^2_{L^2(\Omega\setminus U^\varepsilon;\R^n)}=\frac{\Lambda}{2} \left\|\nabla^s u\right\|^2_{L^2(\Omega\setminus U^\varepsilon;\R^n)}\le\frac{\Lambda}{2}\left\|\nabla v\right\|^2_{L^2(\Omega\setminus K^\varepsilon;\R^n)}\le\frac{\Lambda}{2}\varepsilon.
\end{align}
Finally, for what concerns the first integral in~\eqref{eq:Fhuh}, we observe that, since $\chi^\varepsilon=1$ in $U^\varepsilon$, by~\eqref{eq:nablav} and~\eqref{eq:nablauh-for} we have
\[
\nabla^s u_h^\varepsilon=\nabla v+\nabla(v_h^\varepsilon-v)+R_h^\varepsilon=\nabla v_h^\varepsilon+R_h^\varepsilon\quad\text{a.e.\ in }U^\varepsilon.
\]
Thus,~\eqref{eq:classicbounds} and~\eqref{eq:nablav} imply that
\begin{align}\label{eq:stima}
 &\frac{1}{2}\int_{U^\varepsilon}A_h(x)\nabla^s u_h^\varepsilon(x)\cdot\nabla^s u_h^\varepsilon(x)\,{\rm d}x\nonumber\\
 &=\frac{1}{2}\int_{U^\varepsilon}A_h(x)(\nabla v_h^\varepsilon(x)+R_h^\varepsilon(x)) \cdot(\nabla v_h^\varepsilon(x)+R_h^\varepsilon(x))\,{\rm d}x\nonumber\\
 &=\frac{1}{2}\int_{U^\varepsilon}A_h(x)\nabla v_h^\varepsilon(x) \cdot\nabla v_h^\varepsilon(x)\,{\rm d}x+\int_{U^\varepsilon}A_h(x)\nabla v_h^\varepsilon(x)\cdot R_h^\varepsilon(x)\,{\rm d}x\nonumber\\
 &\quad+\frac{1}{2}\int_{U^\varepsilon}A_h(x)R_h^\varepsilon(x) \cdot R_h^\varepsilon(x)\,{\rm d}x\nonumber\\
 & \le G_h(v_h^\varepsilon)+\int_{U^\varepsilon}A_h(x)\nabla v_h^\varepsilon(x)\cdot R_h^\varepsilon(x)\,{\rm d}x+\frac{1}{2}\int_{U^\varepsilon}A_h(x)R_h^\varepsilon(x) \cdot R_h^\varepsilon(x)\,{\rm d}x
\end{align}
and, since $(v_h^\varepsilon)_h$ is uniformly bounded in $H^1(\Omega)$, by~\eqref{eq:Glimsup},~\eqref{eq:Rh-conv} and~\eqref{eq:stima}, we get
\begin{align}\label{eq:Fh-3}
&\limsup_{h\to\infty}\frac{1}{2}\int_{U^\varepsilon}A_h(x)\nabla^s u_h^\varepsilon(x)\cdot\nabla^s u_h^\varepsilon(x)\,{\rm d}x\le G_\infty(v)+\frac{\Lambda}{2}\varepsilon.
\end{align}
Therefore, by~\eqref{eq:Fh-1},~\eqref{eq:Fh-2} and~\eqref{eq:Fh-3}, we obtain that for all $\varepsilon>0$
\begin{equation}\label{eq:uhepsilon2}
\limsup_{h\to\infty}F_h(u_h^\varepsilon)\le F_\infty(u)+\Lambda\varepsilon.
\end{equation}

To conclude, we use the following diagonal argument. In view of~\cite[Definition~4.1 and Remark~4.3]{DalMaso}, by~\eqref{eq:uhepsilon1} and~\eqref{eq:uhepsilon2}, we have that for all $\varepsilon>0$
\[
\Gamma\text{-}\limsup_{h\to\infty} F_h(u)\coloneqq\sup_{k\in\mathbb N}\limsup_{h\to\infty}\inf_{z\in B_\frac{1}{k}(u)}F_h(z)\le F_\infty(u)+\Lambda\varepsilon.
\]
Hence, by letting $\varepsilon\to 0$, we conclude that
\[
\Gamma\text{-}\limsup_{h\to\infty} F_h(u)\le F_\infty(u)
\]
and, by the properties of the $\Gamma$-$\limsup$ (see e.g.~\cite[Proposition~8.1]{DalMaso}), there exists a sequence $(u_h)_h\subset L^2(\R^n)$ such that $u_h\to u$ strongly in $L^2(\R^n)$ as $h\to\infty$ and 
\[
\limsup_{h\to\infty}F_h(u_h)=\Gamma\text{-}\limsup_{h\to\infty} F_h(u)\le F_\infty(u).
\]
This implies the validity of~\eqref{eq:Fh-Gammalimsup}.
\medskip 

 {\bf Step 2: $\Gamma$-convergence of $(F_h)_h$ implies $\Gamma$-convergence of $(G_h)_h$}. We assume that 
\[
\text{$(F_h)_h$ $\Gamma$-converges to $F_\infty$ strongly in $L^2(\R^n)$},
\]
and we want to show that
\[
\text{$(G_h)_h$ $\Gamma$-converges to $G_\infty$ strongly in $L^2(\Omega)$}.
\]

By Proposition~\ref{prop:DalMaso}, there are a not relabeled subsequence of $(G_h)_h$ and $\widehat G_\infty\colon L^2(\Omega)\to [0,\infty]$ such that 
\begin{equation}\label{eq:hatGh}
\text{$(G_h)_h$ $\Gamma$-converges to $\widehat G_\infty$ strongly in $L^2(\Omega)$}.
\end{equation}
Moreover, there exists $\widehat B_\infty\in\mathcal M^{\rm sym}(\lambda,\Lambda,\Omega)$ such that
\begin{align*}
\widehat G_\infty(v)=\begin{cases}
\displaystyle\frac{1}{2}\int_{\Omega}\widehat B_\infty(x)\nabla v(x)\cdot\nabla v(x)\,{\rm d}x&\text{if $v\in H^1(\Omega)$},\\
\displaystyle\infty&\text{if $v\in L^2(\Omega)\setminus H^1(\Omega)$}.
\end{cases}
\end{align*}
By~\eqref{eq:hatGh} and Step 1, we conclude that
\[
\text{$(F_h)_h$ $\Gamma$-converges to $\widehat F_\infty$ strongly in $L^2(\R^n)$},
\]
where $\widehat F_\infty$ is the nonlocal energies associated with the matrix
\[
\widehat A_\infty(x)\coloneqq
\begin{cases}
\widehat B_\infty(x)&\text{if $x\in\Omega$},\\
A_0(x)&\text{if $x\in\R^n\setminus\Omega$},
\end{cases}
\]
as in~\eqref{eq:Finfty}. 
By the uniqueness of the $\Gamma$-limit, we conclude that $\widehat F_\infty=F_\infty$, and by Lemma~\ref{lem:uniqueness} we derive that $\widehat B_\infty=B_\infty$. Hence, by the Urysohn property of $\Gamma$-convergence, we conclude that 
\[
\text{$(G_h)_h$ $\Gamma$-converges to $G_\infty$ strongly in $L^2(\Omega)$}.
\]
 
\end{proof}

\begin{rmk}
 We point out that the nonlocal energies $(F_h)_h$ account for the boundary condition, while the corresponding local energies $(G_h)_h$ do not, as they are finite in $H^1(\Omega)$ instead of $H^1_0(\Omega)$. We have chosen to work with $(G_h)_h$, since it simplifies the proof of the $\Gamma$-convergence equivalence between nonlocal energies and local ones. On the other hand, if we define
\[
G_h^0(v)\coloneqq
\begin{cases}
\displaystyle\frac{1}{2}\int_{\Omega}B_h(x)\nabla v(x)\cdot\nabla v(x)\,{\rm d}x&\text{if $v\in H^1_0(\Omega)$},\\
\displaystyle\infty&\text{if $v\in L^2(\Omega)\setminus H^1_0(\Omega)$},
\end{cases}
\]
and
\[
G_\infty^0(v)\coloneqq
\begin{cases}
\displaystyle\frac{1}{2}\int_{\Omega}B_\infty(x)\nabla v(x)\cdot\nabla v(x)\,{\rm d}x&\text{if $v\in H^1_0(\Omega)$},\\
\displaystyle\infty&\text{if $v\in L^2(\Omega)\setminus H^1_0(\Omega)$},
\end{cases}
\]
by~\cite[Theorem 13.12 and Theorem 22.4]{DalMaso}, we have that 
\[
(G_h^0)_h\text{ $\Gamma$-converges to }G_\infty^0\text{ strongly in $L^2(\Omega)$}
\]
if and only if
\[
(G_h)_h\text{ $\Gamma$-converges to }G_\infty\text{ strongly in $L^2(\Omega)$}.
\]
Therefore, Proposition~\ref{prop:loc-nonloc} is still valid if we replace $(G_h)_h$ and $G_\infty$, respectively, with $(G_h^0)_h$ and $G_\infty^0$. 
\end{rmk} 

As a consequence of Proposition~\ref{prop:DalMaso} and Proposition~\ref{prop:loc-nonloc}, we can finally prove Theorem~\ref{thm:Gammacompactness}. 

\begin{proof}[Proof of Theorem~\ref{thm:Gammacompactness}]
For every $h\in\mathbb N$, we define the matrix-valued functions
\[
B_h\coloneqq A_h|_\Omega\in\mathcal M^{\rm sym}(\lambda,\Lambda,\Omega),
\]
and we consider the functionals $G_h\colon L^2(\Omega)\rightarrow [0,\infty]$ associated with $B_h$, as in~\eqref{eq:Gh}. 
By Proposition~\ref{prop:DalMaso}, there exist a not relabeled subsequence and a matrix $B_\infty\in\mathcal M^{\rm sym}(\lambda,\Lambda,\Omega)$ such that
\[
\text{$(G_h)_h$ $\Gamma$-converges to $G_\infty$ strongly in $L^2(\Omega)$},
\]
where the functional $G_\infty\colon L^2(\Omega)\rightarrow [0,\infty]$ is defined as in~\eqref{eq:Ginfty}. 

We define 
\[
A_\infty(x)\coloneqq
\begin{cases}
\displaystyle B_\infty(x)&\text{if $x\in\Omega$},\\
\displaystyle A_0(x)&\text{if $x\in\R^n\setminus\Omega$},
\end{cases}
\]
and we let $F_\infty\colon L^2(\Omega)\rightarrow [0,\infty]$ denote the nonlocal energy associated with $A_\infty$, as in~\eqref{eq:Finfty}. 
Then, $A_\infty\in \mathcal M^{\rm sym}(\lambda,\Lambda,\Omega,A_0)$ and, by Proposition~\ref{prop:loc-nonloc}, we deduce that 
\[
\text{$(F_h)_h$ $\Gamma$-converges to $F_\infty$ strongly in $L^2(\R^n)$}.
\]
\end{proof}


 \section{Equivalence between nonlocal \texorpdfstring{$H$}{H}-convergence and \texorpdfstring{$\Gamma$}{Gamma}-convergence}\label{sec:H-Gamma-equiv}

We conclude this paper with the following equivalence between nonlocal $H$-convergence of a sequence $(A_h)_h\subset \mathcal M^{\rm sym}(\lambda,\Lambda,\Omega,A_0)$ and $\Gamma$-convergence of the associated nonlocal energies $(F_h)_h$, introduced in~\eqref{eq:Fh}. 

\begin{theorem}\label{thm:equiv_Gamma-H}
Let $A_0\in\mathcal M^{\rm sym}(\lambda,\Lambda,\R^n)$. For every $h\in\mathbb{N}$, let $A_h,A_\infty \in \mathcal M^{\rm sym}(\lambda,\Lambda,\Omega,A_0)$ and $F_h,F_\infty\colon L^2(\R^n)\to [0,\infty]$ be the nonlocal energies, respectively defined in~\eqref{eq:Fh} and~\eqref{eq:Finfty}.
Then,
\[
\text{$(A_h)_h$ $H$-converges to $A_\infty$ in $H^s_0(\Omega)$}
\]
if and only if
\[
\text{$(F_h)_h$ $\Gamma$-converges to $F_\infty$ strongly in $L^2(\R^n)$}.
\]
\end{theorem}

The proof of Theorem~\ref{thm:equiv_Gamma-H} requires two preliminary results. In Lemma~\ref{lem:equiv_Gamma-G}, we first show the equivalence between the $\Gamma$-convergence of the nonlocal energies and the nonlocal $G$-convergence introduced in Definition~\ref{def:nonlocalG}, which corresponds to the sole convergence of the solutions. Later, in Proposition~\ref{prop:convmom}, we show that the $\Gamma$-convergence of the nonlocal energies also implies the convergence of the momenta, as required in Definition~\ref{def:Hscon}.

The proof of the equivalence between the nonlocal $G$-convergence and the $\Gamma$-convergence of the associated nonlocal energies can be obtained as an application of~\cite[Theorem~13.12]{DalMaso}, with $Y=H^s_0(\Omega)$ and $X=L^2(\R^n)$. For the reader's convenience, we provide below a complete proof, following the techniques presented in~\cite[Theorem~13.5]{DalMaso}.

\begin{lem}\label{lem:equiv_Gamma-G}
For every $h\in\mathbb N$, let $A_h,A_\infty \in \mathcal M^{\rm sym}(\lambda,\Lambda,\R^n)$ and let $F_h,F_\infty\colon L^2(\R^n)\to [0,\infty]$ be the nonlocal energies, respectively defined in~\eqref{eq:Fh} and~\eqref{eq:Finfty}.
Then
\[
\text{$(A_h)_h$ $G$-converges to $A_\infty$ in $H^s_0(\Omega)$}
\]
if and only if
\[
\text{$(F_h)_h$ $\Gamma$-converges to $F_\infty$ strongly in $L^2(\R^n)$}.
\]
\end{lem}

\begin{proof}
{\bf $\Gamma$-convergence implies $G$-convergence.}

We assume that $(F_h)_h$ $\Gamma$-converges to $F_\infty$ strongly in $L^2(\R^n)$ and we first show that~\eqref{eq:Gscon} holds for every $f\in L^2(\R^n)$.
Later, by a density argument, we extend the validity of~\eqref{eq:Gscon} for all $f\in H^{-s}(\Omega)$, which implies the $G$-convergence of $(A_h)_h$ to $A_\infty$ in $H^s_0(\Omega)$.

{\bf Step 1.} We fix $g\in L^2(\R^n)$ and, for every $h\in\mathbb N$, we define $F_h^g,F_\infty^g\colon L^2(\R^n)\to [0,\infty]$ as
\begin{equation}\label{eq:Fhf}
F_h^g(u)\coloneqq F_h(u)+\int_{\R^n} g(x)u(x)\,{\rm d} x\quad\text{and}\quad F_\infty^g(u)\coloneqq F_\infty(u)+\int_{\R^n} g(x)u(x)\,{\rm d} x
\end{equation}
for all $u\in L^2(\R^n)$.
Since we have perturbed with continuity $F_h$ and $F_\infty$, then
\[
\text{$(F_h^g)_h$ $\Gamma$-converges to $F_\infty^g$ strongly in $L^2(\R^n)$},
\]
in virtue of~\cite[Proposition~6.21]{DalMaso}.

We note that the solutions $w_h\in H^s_0(\Omega)$ and $w_\infty\in H^s_0(\Omega)$ of problems $(P_h^g)$ and $(P_\infty^g)$, whose existence is guaranteed by Lemma~\ref{lem:LaxMilgram}, minimise the energies $F_h^g$ and $F_\infty^g$, respectively, i.e.
\[
F_h^g(w_h)=\min_{u\in L^2(\R^n)}F_h^g(u)\quad\text{and}\quad F_\infty^g(w_\infty)=\min_{u\in L^2(\R^n)}F_\infty^g(u).
\]
Therefore, by the Fundamental Theorem of $\Gamma$-convergence (see e.g.~\cite[Theorem~7.8]{DalMaso}), we get
\[
\text{$w_h\to w_\infty$ strongly in $L^2(\R^n)$ as $h\to\infty$}.
\]
In particular, since $(w_h)_h\subset H^s_0(\Omega)$ is uniformly bounded, we conclude that 
\begin{equation}\label{eq:whjwh}
\text{$w_h\to w_\infty$ weakly in $H^s_0(\Omega)$ as $h\to\infty$}.
\end{equation}

{\bf Step 2.} We fix now $f\in H^{-s}(\Omega)$ and we let $u_h,u_\infty\in H^s_0(\Omega)$ denote the solutions of the problems $(P_h^f)$ and $(P_\infty^f)$, respectively.

Since the embedding $H^s_0(\Omega)\subset L^2(\R^n)$ is continuous and dense, so does the embedding $L^2(\R^n)\subset H^{-s}(\Omega)$. Therefore we can find a sequence $(f_j)_j\subset L^2(\R^n)$ such that
\[
\text{$f_j\to f$ strongly in $H^{-s}(\Omega)$ as $j\to\infty$}.
\]

For all $j\in\mathbb N$, let $u_h^j,u_\infty^j\in H^s_0(\Omega)$ be the solutions of problems $(P_h^{f_j})$ and $(P_\infty^{f_j})$, respectively. Fixed $g\in H^{-s}(\Omega)$, by Proposition~\ref{prop:Poincare} and Lemma~\ref{lem:LaxMilgram}, we obtain
\begin{align*}
&|\langle g,u_h-u_\infty\rangle_{H^{-s}(\Omega)\times H^s_0(\Omega)}|\le |\langle g,u_h^j-u_\infty^j\rangle_{H^{-s}(\Omega)\times H^s_0(\Omega)}|+C\|g\|_{H^{-s}(\Omega)}\|f_j-f\|_{H^{-s}(\Omega)}.
\end{align*}
Hence, in view of~\eqref{eq:whjwh}, by letting first $h\to\infty$ and then $j\to\infty$, we obtain~\eqref{eq:Gscon}, and so
\[
\text{$(A_h)_h$ $G$-converges to $A_\infty$ in $H^s_0(\Omega)$}.
\]

{\bf $G$-convergence implies $\Gamma$-convergence.}
We now show that 
\[
\text{$(F_h)_h$ $\Gamma$-converges to $F_\infty$ strongly in $L^2(\R^n)$}, 
\]
in accordance with Definition~\ref{def:Gamma}, by assuming the $G$-convergence of the associated matrices.

{\bf $\Gamma$-liminf inequality}. 
We fix $u_h,u\in L^2(\R^n)$, $h\in\mathbb{N}$, such that $u_h\to u$ strongly in $L^2(\R^n)$ as $h\to\infty$ and, up to a not relabeled subsequence, we assume that
\[
\lim_{h\to\infty}F_h(u_h)=\liminf_{h\to\infty}F_h(u_h)<\infty\quad\text{and}\quad\sup_{h\in\mathbb N}F_h(u_h)\le C,
\]
for some positive constant $C$, the conclusion being otherwise trivial.

Then, $(u_h)_h\subset H^s_0(\Omega)$ and there exists another positive constant $C$ such that 
\[
\|u_h\|_{H^s_0(\Omega)}\le C\quad\text{for all $h\in\mathbb N$},
\]
which gives that also $u\in H^s_0(\Omega)$ and
\[
u_h\to u\quad\text{weakly in $H^s_0(\Omega)$ as $h\to\infty$}.
\]

For every $h\in\mathbb{N}$, define
\begin{equation}\label{eq:sourceterm}
f\coloneqq -\div^s(A_\infty\nabla^s u)\in H^{-s}(\Omega)
\end{equation}
and consider the unique weak solution $w_h\in H^s_0(\Omega)$ of problem $(P_h^f)$.
Since 
\[
\text{$(A_h)_h$ $G$-converges to $A_\infty$ in $H^s_0(\Omega)$}
\]
and, by construction, $u$ is the unique weak solution of $(P_\infty^f)$, it holds that
\[
w_h\to u\quad\text{weakly in $H^s_0(\Omega)$ as $h\to\infty$}.
\]
From the one hand, we have that
\[
\lim_{h\to\infty}\frac{1}{2}\int_{\R^n} A_\infty(x)\nabla^s u(x)\cdot (2\nabla^s u_h(x)-\nabla^s w_h(x))\,{\rm d} x=F_\infty(u).
\]
From the other hand, by~\eqref{eq:sourceterm}
\begin{align*}
&\int_{\R^n} A_\infty(x)\nabla^s u(x)\cdot (2\nabla^s u_h(x)-\nabla^s w_h(x))\,{\rm d} x\\
&=\langle f,2u_h-w_h\rangle_{H^{-s}(\Omega)\times H^s_0(\Omega)}=\int_{\R^n} A_h(x)\nabla^s w_h(x)\cdot (2\nabla^s u_h(x)-\nabla^s w_h(x))\,{\rm d} x\\
&\le \int_{\R^n} A_h(x)\nabla^s u_h(x)\cdot\nabla^s u_h(x)\,{\rm d} x=2F_h(u_h),
\end{align*}
being
$$
A_h(x)\xi_1\cdot\xi_1-A_h(x)\xi_2\cdot (2\xi_1-\xi_2)=A_h(x)(\xi_1-\xi_2)\cdot(\xi_1-\xi_2)\ge 0
$$
for a.e.\ $x\in{\R^n}$ and all $\xi_1,\xi_2\in\R^n$.

Hence,
$$
\liminf_{h\to\infty}F_h(u_h)=\lim_{h\to\infty}F_h(u_h)\ge \lim_{h\to\infty}\frac{1}{2}\int_{\R^n} A_\infty(x)\nabla^s u(x)\cdot (2\nabla^s u_h(x)-\nabla^s w_h(x))\,{\rm d} x=F_\infty(u),
$$
which conclude the proof of the $\Gamma$-liminf inequality.
\smallskip

{\bf $\Gamma$-limsup inequality}. We fix $u\in H^s_0(\Omega)$, the conclusion being otherwise trivial by the definition of $F_\infty$, we set $f\coloneqq -\div(A_\infty \nabla^s u)$, and we consider $w_h\in H^s_0(\Omega)$, the unique weak solution of $(P_h^f)$.
By the $G$-convergence of $(A_h)_h$ towards $A_\infty$ in $H^s_0(\Omega)$ and, by Proposition~\ref{prop:Rellich},
\[
w_h\to u\quad\text{weakly in $H^s_0(\Omega)$}\quad\text{and}\quad w_h\to u\quad\text{strongly in $L^2(\R^n)$ as $h\to\infty$}.
\]
Moreover,
\begin{align*}
\lim_{h\to\infty}F_h(w_h)&=\lim_{h\to\infty}\frac{1}{2}\int_{\R^n} A_h(x)\nabla^s w_h(x)\cdot \nabla^s w_h(x)\,{\rm d} x\\
&=\lim_{h\to\infty}\frac{1}{2}\langle f, w_h\rangle_{H^{-s}(\Omega)\times H^s_0(\Omega)}=\frac{1}{2}\langle f,u\rangle_{H^{-s}(\Omega)\times H^s_0(\Omega)}\\
&=\frac{1}{2}\int_{\R^n} A_\infty(x)\nabla^s u(x)\cdot \nabla^s u(x)\,{\rm d} x=F_\infty(u).
\end{align*}
Hence the $\Gamma$-limsup inequality also holds (the limit is actually achieved) and this concludes the proof of the $\Gamma$-convergence of $(F_h)_h$ towards $F_\infty$.
\end{proof}

\begin{rmk}
If we consider only the equivalence between the nonlocal $G$-convergence and the $\Gamma$-convergence of the associated nonlocal energies $(F_h)_h$, then the assumption of fixing the matrices $(A_h)_h$ outside the reference domain $\Omega$ may be omitted. 
On the other hand, for the compactness of the $\Gamma$-convergence (see Theorem~\ref{thm:Gammacompactness}) and for the following equivalence between the nonlocal $H$-convergence and the $\Gamma$-convergence, we need to consider the subclass $\mathcal M^{\rm sym}(\lambda,\Lambda,\Omega,A_0)$, for a given $A_0\in\mathcal M^{\rm sym}(\lambda,\Lambda,\R^n)$. 
\end{rmk}

In view of Lemma~\ref{lem:equiv_Gamma-G}, to prove Theorem~\ref{thm:equiv_Gamma-H} it is sufficient to show that the $\Gamma$-convergence of the nonlocal energies $(F_h)_h$ to $F_\infty$ also implies the convergence of momenta, as required in Definition~\ref{def:Hscon}. To this aim, we follow the strategies adopted in~\cite[Lemma~4.11]{DalMasoCrack} and~\cite[Theorem~4.5]{ADMZ}.
We define the functionals $\mathcal F_h,\mathcal F_\infty\colon L^2(\R^n;\R^n)\to\mathbb{R}$, respectively, as
\begin{align*}
 \mathcal F_h(\Phi)&\coloneqq\frac{1}{2}\int_{\R^n}A_h(x)\Phi(x)\cdot\Phi(x)\,{\rm d}x\quad\text{for all $\Phi\in L^2(\R^n;\R^n)$},\\
 \mathcal F_\infty(\Phi)&\coloneqq\frac{1}{2}\int_{\R^n}A_\infty(x)\Phi(x)\cdot\Phi(x)\,{\rm d}x\quad\text{for all $\Phi\in L^2(\R^n;\R^n)$},
\end{align*}
and consider their Fréchet derivatives $\mathcal{F}_h'$ and $\mathcal F_\infty'$, which are given by
\begin{align*}
 \mathcal F_h'(\Phi)[\Psi]=\int_{\R^n}A_h(x)\Phi(x)\cdot\Psi(x)\,{\rm d}x\quad\text{and}\quad \mathcal F_\infty'(\Phi)[\Psi]=\int_{\R^n}A_\infty(x)\Phi(x)\cdot\Psi(x)\,{\rm d}x
\end{align*}
for all $\Phi,\Psi\in L^2(\R^n;\R^n)$.

Note that $\mathcal F_h'$ and $\mathcal F_\infty'$ identify the momenta for the functionals $F_h$ and $F_\infty$, respectively. 
Indeed, given $(u_h)_h\subset H^s_0(\Omega)$ and $u_\infty \in H^s_0(\Omega)$, then the convergence
\[
\mathcal F_h'(\nabla^s u_h)[\Psi]\to\mathcal F_\infty'(\nabla^s u_\infty)[\Psi]\quad\text{for all $\Psi\in L^2(\Omega;\R^n)$}
\]
is equivalent to
\[
A_h\nabla^s u_h\to A_\infty\nabla^s u_\infty\quad\text{weakly in $ L^2(\R^n;\R^n)$}.
\]
We have the following result.

\begin{prop}\label{prop:convmom} 
Let $A_0\in\mathcal M^{\rm sym}(\lambda,\Lambda,\R^n)$. For every $h\in\mathbb N$, let $A_h,A_\infty \in \mathcal M^{\rm sym}(\lambda,\Lambda,\Omega,A_0)$ and $F_h,F_\infty\colon L^2(\R^n)\to [0,\infty]$ be the nonlocal energies, respectively defined in~\eqref{eq:Fh} and~\eqref{eq:Finfty}.
Assume that
\[
\text{$(F_h)_h$ $\Gamma$-converges to $F_\infty$ strongly in $L^2(\R^n)$},
\]
 and let $(u_h)_h\in H^s_0(\Omega)$ and $u_\infty \in H^s_0(\Omega)$ satisfy
\begin{align}\label{eq:recovery}
u_h\to u_\infty\quad\text{strongly in $L^2(\R^n)$ as $h\to\infty$}\quad\text{and}\quad F_h(u_h)\to F_\infty(u_\infty)\quad\text{as $h\to\infty$}.
\end{align}
Then, the convergence of the momenta associated with $(F_h)_h$ and $F_\infty$ holds, i.e.
\begin{align}\label{eq:gateaux}
\mathcal F_h'(\nabla^s u_h)[\Psi]\to\mathcal F_\infty'(\nabla^s u_\infty)[\Psi]\quad\text{for all $\Psi\in L^2(\R^n;\R^n)$ as $h\to\infty$}.
\end{align}
\end{prop}

\begin{proof}
To prove~\eqref{eq:gateaux}, it is sufficient to show the following inequality
\begin{equation}\label{eq:thesismom}
\mathcal F'(\nabla^s u_\infty)[\Psi]\le\,\liminf_{h\to\infty}\mathcal F'_h(\nabla^s u_h)[\Psi]\quad\text{for all $\Psi\in L^2(\R^n;\R^n)$}.
\end{equation}
Indeed, by replacing $\Psi$ with $-\Psi$, and by the properties of the limit inferior, one can get the desired condition~\eqref{eq:gateaux}.

 For every $h\in\mathbb N$, we define the matrix-valued functions
\[
B_h\coloneqq A_h|_\Omega\in\mathcal M^{\rm sym}(\lambda,\Lambda,\Omega),\qquad B_\infty\coloneqq A_\infty|_\Omega\in\mathcal M^{\rm sym}(\lambda,\Lambda,\Omega),
\]
and the functionals $G_h,G_\infty\colon L^2(\Omega)\rightarrow [0,\infty]$, which are the local energies associated with $B_h$ and $B_\infty$, as in~\eqref{eq:Gh} and~\eqref{eq:Ginfty}, respectively. Thanks to Proposition~\ref{prop:loc-nonloc}, we have
\begin{equation}\label{eq:GhGamma}
\text{$(G_h)_h$ $\Gamma$-converges to $G_\infty$ strongly in $L^2(\Omega)$.}
\end{equation}
 For all $\Phi\in L^2(\Omega;\R^n)$, we define $\mathcal G_h^\Phi,\mathcal G_\infty^\Phi\colon L^2(\Omega)\to [0,\infty]$ as
\begin{align*}
&\mathcal G_h^\Phi(v)\coloneqq
\begin{cases}
\displaystyle\frac{1}{2}\int_{\Omega}A_h(x)(\nabla v(x)+\Phi(x))\cdot(\nabla v(x)+\Phi(x))\,{\rm d}x&\text{if $v\in H^1(\Omega)$},\\
\displaystyle\infty&\text{if $v\in L^2(\Omega)\setminus H^1(\Omega)$},
\end{cases}\\
&\mathcal G_\infty^\Phi(v)\coloneqq
\begin{cases}
\displaystyle\frac{1}{2}\int_{\Omega}A_\infty(x)(\nabla v(x)+\Phi(x))\cdot(\nabla v(x)+\Phi(x))\,{\rm d}x&\text{if $v\in H^1(\Omega)$},\\
\displaystyle\infty&\text{if $v\in L^2(\Omega)\setminus H^1(\Omega)$}.
\end{cases}
\end{align*}
By~\cite[Theorem~22.4]{DalMaso} and~\cite[Theorem~4.2]{ADMZ},~\eqref{eq:GhGamma} implies that for all $\Phi\in L^2(\Omega;\R^n)$
\begin{equation}\label{eq:gammaGPhi}
(\mathcal G_h^\Phi)_h\text{ $\Gamma$-converges to }\mathcal G_\infty^\Phi\text{ strongly in $L^2(\Omega)$}.
\end{equation}

Let $\Psi\in L^2(\R^n;\R^n)$ and $(t_i)_i$ be a sequence of positive numbers, infinitesimal as $i\to\infty$. 
Moreover, let $u_h,u_\infty\in H^s_0(\Omega)$, $h\in\mathbb N$, satisfy~\eqref{eq:recovery}, and define
\begin{align*}
v_h\coloneqq I_{1-s}u_h\in H^1(\Omega)\quad\text{and}\quad v_\infty\coloneqq I_{1-s}u_\infty\in H^1(\Omega). 
\end{align*}
By the continuity of $I_{1-s}\colon L^2(\R^n)\to L^2(\Omega)$, it holds that 
\[
v_h\to v_\infty\quad\text{strongly in $L^2(\Omega)$ as $h\to\infty$}
\]
and, by~\eqref{eq:gammaGPhi} with $\Phi\coloneqq t_i\Psi|_\Omega$, we get
\begin{align}\label{eq:mathcalGh-Gammaliminf}
\mathcal G^{t_i\Psi|_\Omega}_\infty(v_\infty)\le\liminf_{h\to\infty}\mathcal G^{t_i\Psi|_\Omega}_h(v_h)\quad\text{for all }i\in\mathbb N.
\end{align}
Since, by Proposition~\ref{prop:Riesz},
\[
\nabla v_h=\nabla^s u_h\quad\text{and}\quad\nabla v_\infty=\nabla^s u_\infty\quad\text{a.e.\ in $\R^n$},
\]
we can then rephrase~\eqref{eq:mathcalGh-Gammaliminf} as
\begin{align}
&\frac{1}{2}\int_\Omega A_\infty(x)(\nabla^s u_\infty(x)+t_i\Psi(x))\cdot(\nabla^s u_\infty(x)+t_i\Psi(x))\,{\rm d}x\nonumber\\
&\le\liminf_{h\to\infty}\frac{1}{2}\int_\Omega A_h(x)(\nabla^s u_h(x)+t_i\Psi(x))\cdot(\nabla^s u_h(x)+t_i\Psi(x))\,{\rm d}x.\label{eq:stima1}
\end{align}
In addition, by~\eqref{eq:unif-bound1} and~\eqref{eq:recovery}, the sequence $(u_h)_h$ is uniformly bounded in $H^s_0(\Omega)$ and, by the strong convergence of $(u_h)_h$ to $u$ in $L^2(\R^n)$, we have
\[
\nabla^s u_h+t_i\Psi\to\nabla^s u_\infty+t_i\Psi\quad\text{weakly in $L^2(\R^n;\R^n)$ as }h\to\infty.
\]
Then,
\begin{align}
&\frac{1}{2}\int_{\R^n\setminus\Omega} A_0(x)(\nabla^s u_\infty(x)+t_i\Psi(x))\cdot(\nabla^s u_\infty(x)+t_i\Psi(x))\,{\rm d}x\nonumber\\
&\le\liminf_{h\to\infty}\frac{1}{2}\int_{\R^n\setminus\Omega} A_0(x)(\nabla^s u_h(x)+t_i\Psi(x))\cdot(\nabla^s u_h(x)+t_i\Psi(x))\,{\rm d}x.\label{eq:stima2}
\end{align}
Therefore, by~\eqref{eq:stima1} and~\eqref{eq:stima2},
\begin{align*}
\mathcal F_\infty(\nabla^s u_\infty+t_i\Psi)&=\frac{1}{2}\int_\Omega A_\infty(x)(\nabla^s u_\infty(x)+t_i\Psi(x))\cdot(\nabla^s u_\infty(x)+t_i\Psi(x))\,{\rm d}x\\
&\quad+\frac{1}{2}\int_{\R^n\setminus\Omega} A_0(x)(\nabla^s u_\infty(x)+t_i\Psi(x))\cdot(\nabla^s u_\infty(x)+t_i\Psi(x))\,{\rm d}x\\
&\le\liminf_{h\to\infty}\frac{1}{2}\int_\Omega A_h(x)(\nabla^s u_h(x)+t_i\Psi(x))\cdot(\nabla^s u_h(x)+t_i\Psi(x))\,{\rm d}x\\
&\quad+\liminf_{h\to\infty}\frac{1}{2}\int_{\R^n\setminus \Omega} A_0(x)(\nabla^s u_h(x)+t_i\Psi(x))\cdot(\nabla^s u_h(x)+t_i\Psi(x))\,{\rm d}x\\
&\le\liminf_{h\to\infty}\mathcal F_h(\nabla^s u_h+t_i\Psi).
\end{align*}
Moreover, since by definition
\[
\mathcal F_h(\nabla^s u_h)=F_h(u_h)\quad\text{and}\quad\mathcal F_\infty(\nabla^s u_\infty)=F_\infty(u_\infty),
\]
then, by~\eqref{eq:recovery}, it holds that
\[
\frac{\mathcal F_\infty(\nabla^s u_\infty+t_i\Psi)-\mathcal F_\infty(\nabla^s u_\infty)}{t_i}\le\liminf_{h\to\infty}\frac{\mathcal F_h\left(\nabla^s u_h+t_i\Psi\right)-\mathcal F_h\left(\nabla^s u_h\right)}{t_i}\quad\text{for all $i\in\mathbb N$}.
\]
Hence, there exists an increasing sequence of integers $(h_i)_i\subset \mathbb N$ such that
\begin{align}
\frac{\mathcal F_\infty\left(\nabla^s u_\infty+t_i\Psi\right)-\mathcal F_\infty\left(\nabla^s u_\infty\right)}{t_i}-\frac{1}{i}\le\frac{\mathcal F_h\left(\nabla^s u_h+t_i\Psi\right)-\mathcal F_h\left(\nabla^s u_h\right)}{t_i}\quad\text{for all $h\ge h_i$}.\label{eq:convmom2}
\end{align}
If we set $\varepsilon_h\coloneqq t_i$ for $h_i\le h<h_{i+1}$ and $i\in\mathbb N$, then, by~\eqref{eq:convmom2}
\begin{align}
 &\liminf_{h\to\infty}\frac{\mathcal F_\infty\left(\nabla^s u_\infty+\varepsilon_h\Psi\right)-\mathcal F_\infty\left(\nabla^s u_\infty\right)}{\varepsilon_h}\le\liminf_{h\to\infty}\frac{\mathcal F_h\left(\nabla^s u_h+\varepsilon_h\Psi\right)-\mathcal F_h\left(\nabla^s u_h\right)}{\varepsilon_h}.\label{eq:mom1}
\end{align}
Note that the limit inferior on the left-hand side of~\eqref{eq:mom1} is actually achieved and coincides with the Fréchet derivative of the functional $\mathcal F_\infty$, i.e.\
\begin{equation}\label{eq:mom2}
 \mathcal F'_\infty(\nabla^s u_\infty)[\Psi]=\lim_{h\to\infty}\frac{\mathcal F_\infty\left(\nabla^s u_\infty+\varepsilon_h\Psi\right)-\mathcal F_\infty\left(\nabla^s u_\infty\right)}{\varepsilon_h}.
\end{equation}
For what concerns the right-hand side of~\eqref{eq:mom1}, we have
\begin{align}\label{eq:mom3}
 \frac{\mathcal F_h\left(\nabla^s u_h+\varepsilon_h\Psi\right)-\mathcal F_h\left(\nabla^s u_h\right)}{\varepsilon_h}=\mathcal F_h'(\nabla^s u_h)[\Psi]+\varepsilon_h\mathcal F_h(\Psi).
\end{align}

Since the last term on the right-hand side of~\eqref{eq:mom3} converges to $0$ as $h\to\infty$, from~\eqref{eq:mom1}--\eqref{eq:mom3} we get~\eqref{eq:thesismom}.
\end{proof}

We can finally prove Theorem~\ref{thm:equiv_Gamma-H}.

\begin{proof}[Proof of Theorem~\ref{thm:equiv_Gamma-H}]

 We need to prove that the $H$-convergence of $(A_h)_h$ to $A_\infty$ in $H^s_0(\Omega)$ implies the $\Gamma$-convergence of $(F_h)_h$ in the strong topology of $L^2(\R^n)$, and viceversa.

\smallskip

{\bf $H$-convergence implies $\Gamma$-convergence.} Since the $H$-convergence is stronger than the $G$-convergence, this part follows by Lemma~\ref{lem:equiv_Gamma-G}.

\smallskip

 {\bf $\Gamma$-convergence implies $H$-convergence.}
As in Lemma~\ref{lem:equiv_Gamma-G}, we first show that~\eqref{eq:Hscon1} and~\eqref{eq:Hscon2} hold for every $f\in L^2(\R^n)$. Then, by a density argument, we show that~\eqref{eq:Hscon1} and~\eqref{eq:Hscon2} are satisfied for all $f\in H^{-s}(\Omega)$, leading to the $H$-convergence of $(A_h)_h$ to $A_\infty$.
\smallskip

{\bf Step 1.} We fix $f\in L^2(\R^n)$. By Lemma~\ref{lem:equiv_Gamma-G}, the $\Gamma$-convergence of $(F_h)_h$ towards $F_\infty$ implies the weak convergence in $H^s_0(\Omega)$ of $u_h$, solutions of the problem $(P_h^f)$, towards $u_\infty$, solution of the limit problem $(P_\infty^f)$. 
Moreover, $u_h$ and $u_\infty$ are also, respectively, the unique minimisers of the nonlocal energies $F_h^f$ and $F_\infty^f$ defined in~\eqref{eq:Fhf}. Hence, by the Fundamental Theorem of $\Gamma$-convergence~\cite[Theorem~7.8]{DalMaso}, $F_h^f(u_h)\to F_\infty^f(u_\infty)$ as $h\to\infty$, which implies that
\[
F_h(u_h)\to F_\infty(u_\infty)\quad\text{as $h\to\infty$}.
\]
Therefore, in virtue of Proposition~\ref{prop:convmom}, it holds that
\[
A_h\nabla^s u_h\to A_\infty \nabla^s u_\infty\quad\text{weakly in $L^2(\R^n;\R^n)$ as $h\to\infty$},
\]
and~\eqref{eq:Hscon1}--\eqref{eq:Hscon2} are satisfied for all $f\in L^2(\R^n)$.
\smallskip

{\bf Step 2.} We fix now $f\in H^{-s}(\Omega)$.
For every $h\in\mathbb N$, let $u_h,u_\infty\in H^s_0(\Omega)$ be the solutions of problems $(P_h^f)$ and $(P^f_\infty)$, respectively.
By Lemma~\ref{lem:equiv_Gamma-G}, we already know the validity of the convergence of the solutions~\eqref{eq:Hscon1}, and it remains to prove the convergence of the momenta~\eqref{eq:Hscon2}.

Since the embedding $L^2(\R^n)\subset H^{-s}(\Omega)$ is continuous and dense, there exists a sequence $(f_j)_j\subset L^2(\R^n)$ such that as $j\to\infty$
\[
\text{$(f_j)_j$ strongly converges to $f$ in $H^{-s}(\Omega)$.}
\]
For all $j\in\mathbb N$, let $u_h^j,u_\infty^j\in H^s_0(\Omega)$ be, respectively, the solutions of the problems $(P_h^{f_j})$ and $(P_\infty^{f_j})$.
Fixed $\Phi\in L^2(\R^n;\R^n)$, by Proposition~\ref{prop:Poincare} and Lemma~\ref{lem:LaxMilgram}, we get
\begin{align*}
&\left|\int_\Omega (A_h(x)\nabla^su_h(x)-A_\infty(x)\nabla^s u_\infty(x))\cdot \Phi(x)\,{\rm d}x\right|\\
&\le \left|\int_\Omega (A_h(x)\nabla^su_h^j(x)-A_\infty(x)\nabla^s u_\infty^j(x))\cdot \Phi(x)\,{\rm d}x\right|+C\|\Phi\|_{L^2(\R^n;\R^n)}\|f_j-f\|_{H^{-s}(\Omega)},
\end{align*}
and by Step 1, by sending first $h\to\infty$ and then $j\to\infty$, we obtain~\eqref{eq:Hscon2}, leading to the $H$-convergence of the sequence $(A_h)_h$ to $A_\infty$ in $H^s_0(\Omega)$.
\end{proof}

 We conclude this section providing an alternative proof of Theorem~\ref{thm:sym-compact}, purely based on the variational techniques introduced in Theorem~\ref{thm:Gammacompactness} and Theorem~\ref{thm:equiv_Gamma-H}.

\begin{proof}[Proof of Theorem~\ref{thm:sym-compact}]
Let $(A_h)_h\subset\mathcal M^{\rm sym}(\lambda,\Lambda,\Omega,A_0)$ and let $(F_h)_h$ be the associated nonlocal energies, as in~\eqref{eq:Fh}.
By Theorem~\ref{thm:Gammacompactness}, there exists a not relabeled subsequence of $(A_h)_h$ and $A_\infty\in \mathcal M^{\rm sym}(\lambda,\Lambda,\Omega,A_0)$ such that, denoting by $F_\infty$ the nonlocal energy associated with $A_\infty$, as in~\eqref{eq:Finfty}, then
\[
\text{$(F_h)_h$ $\Gamma$-converges to $F_\infty$ strongly in $L^2(\R^n)$}. 
\]
 We can therefore conclude that
\[
\text{$(A_h)_h$ $H$-converges to $A_\infty$ in $H^s_0(\Omega)$},
\]
in virtue of Theorem~\ref{thm:equiv_Gamma-H}. 
\end{proof}


\section{Conclusions and open problems}\label{sec:conclusions}

Through our distributional approach, the $H$-convergence theory extends to linear operators in fractional divergence form.
In what follows, we list some possible future research directions stemming from our results that we believe may be of particular interest to the community.
\begin{enumerate}
 \item 
 A key tool useful to characterise the $H$-limit for sequence of \emph{monotone operators} is the Div-Curl Lemma~\cite{MuratI,MuratII}. 
 It is worth investigating the validity of an analogous of this result in the fractional setting.
 \item In Section~\ref{sec:H-Gamma-equiv}, we proved that the $H$-compactness in the symmetric case is equivalent to the $\Gamma$-compactness of the associated energies. In~\cite{ADMZ2}, the authors show that an analogous result can be obtained also in the case of not necessarily symmetric matrices for which, a priori, there is no natural energy associated with the problem. We conjecture that the techniques used in the aforementioned paper can be adapted in the fractional scenario.
 \item Once the $H$-convergence for elliptic operators has been characterised, it is natural to study the asymptotic behaviour of sequences of parabolic nonlocal operators of the form 
 \[
 \partial_t-\div^s(B_h(x)\nabla^s).
 \]
 In~\cite{MPV}, the authors show that, whenever the sequence of matrix-valued functions $(A_h)_h$ is independent of time, then the parabolic $H$-limit $B_\infty(x,t)$ coincides with the elliptic $H$-limit $A_\infty(x)$, meaning that $B_\infty$ is constant in time.
 Again, the authors conjecture that a similar discussion can be extended to the nonlocal scenario.
 \item The notion of $H$-convergence has been recently extended to sub-Riemannian structures and, more generally, to operators defined via families of vector fields, see e.g.~\cite{Maione,MPV,MPV2,MPSC,MPSC2}. 
 However, in such contexts the definition and analysis of nonlocal operators is still in its early stages. We believe that a natural continuation of this line of research is to investigate the compactness properties of fractional-order operators in specific sub-Riemannian geometries, such as Carnot groups, which are a well-structured setting for nonlocal analysis.
\end{enumerate}


\section*{Declarations}

\noindent{\it \textbf{Acknowledgements.}} The authors warmly thank G.~C.~Brusca, G.~Dal Maso, C.~Kreisbeck, F.~Paronetto, F.~Solombrino, E.~Valdinoci, and the anonymous referee for their careful reading, for useful suggestions and discussions on this topic.

\smallskip

\noindent{\it \textbf{Conflict of interest.}} On behalf of all authors, the corresponding author states that there is no conflict of interest.

\smallskip

\noindent{\it \textbf{Funding information.}} The authors are member of GNAMPA of the Istituto Nazionale di Alta Matematica (INdAM).

M.C.\ has been founded by the European Union - NextGenerationEU under the Italian Ministry of University and Research (MUR) National National Centre for HPC, Big Data and Quantum Computing (CN\_00000013 – CUP: E13C22001000006). M.C.\ acknowledges also the support of the MUR - PRIN 2022 project ``Variational Analysis of Complex Systems in Materials Science, Physics and Biology'' (N.~2022HKBF5C), funded by European Union NextGenerationEU, and of the INdAM - GNAMPA 2025 Project ``DISCOVERIES - Difetti e Interfacce in Sistemi Continui: un’Ottica Variazionale in Elasticità con Risultati Innovativi ed Efficaci Sviluppi'' (Grant Code: CUP\_E5324001950001). 

A.C.\ acknowledges the support of the MUR - PRIN 2022 project ``Elliptic and parabolic problems, heat kernel estimates and spectral theory'' (N.~20223L2NWK), of the INdAM - GNAMPA 2024 Project ``Ottimizzazione e disuguaglianze funzionali per problemi geometrico-spettrali locali e nonlocali'' (Grant Code: CUP\_E53C23001670001).

A.M.\ is supported by the Spanish State Research Agency, through the Severo Ochoa and María de Maeztu Program for Centers and Units of Excellence in R\&D (CEX2020-001084-M), by MCIN/AEI/10.13039/501100011033 (PID2021-123903NB-I00), by Generalitat de Catalunya (2021-SGR-00087). 

M.C.\ and A.M.\ acknowledge the support of the INdAM - GNAMPA 2024 Project ``Pairing and div-curl lemma: extensions to weakly differentiable vector fields and nonlocal differentiation'' (Grant Code: CUP\_E53C23001670001).

A.C.\ and A.M.\ acknowledge the support of the INdAM - GNAMPA 2025 Project ``Metodi variazionali per problemi dipendenti da operatori frazionari isotropi e anisotropi'' (Grant Code: CUP\_E5324001950001).

This research was supported by the Centre de Recerca Matemàtica of Barcelona (CRM), under the International Programme for Research in Groups (IP4RG).

\smallskip

\noindent{\it \textbf{Data availability statement.}} Data sharing not applicable to this article as no datasets were generated or analysed during the current study.



\begin{thebibliography}{99}

\bibitem{ABSS23}
{\sc R.~Alicandro, A.~Braides, M.~Solci, and G.~Stefani}, 
{\em Topological singularities arising from fractional-gradient energies}, 
Accepted paper, to appear on Math. Ann., Preprint available at arXiv: \href{https://arxiv.org/abs/2309.10112}{2309.10112}, (2025).

\bibitem{ADMZ2}
{\sc N.~Ansini, G.~Dal Maso, and C.~I.~Zeppieri},
{\em $\Gamma$-convergence and $H$-convergence of linear elliptic operators},
J.\ Math.\ Pures Appl.\ (9) {\bf 99}, 321--329, (2013).

\bibitem{ADMZ}
{\sc N.~Ansini, G.~Dal Maso, and C.~I.~Zeppieri},
{\em New results on $\Gamma$-limits of integral functionals},
Ann.\ Inst.\ Henri Poincaré, Anal.\ Non Linéaire {\bf 31}, 185--202, (2014).

\bibitem{BE}
{\sc J.~C.~Bellido and A.~Evgrafov},
{\em A simple characterization of {$H$}-convergence for a class of nonlocal problems}, 
Rev.\ Mat.\ Complut.\ {\bf 34}, 175--183, (2021). 

\bibitem{Bra}
{\sc A.~Braides}, 
$\Gamma$-convergence for beginners,
Oxford Lecture Series in Mathematics and its Applications 22. 
Oxford University Press, Oxford, 2002.

\bibitem{BraBruDo}
{\sc A.~Braides, G.~C.~Brusca, and D.~Donati}, 
{\em Another look at elliptic homogenization}, 
Milan J. Math. \textbf{92}, 1--23, (2024).

\bibitem{BraidesDalMaso}
{\sc A.~Braides and G.~Dal Maso}, 
{\em Validity and failure of the integral representation of $\Gamma$-limits of convex non-local functionals},
J.\ Funct.\ Anal.\ {\bf 286}, Paper No.\ 110317, (2024).

\bibitem{BCCS22}
{\sc E.~Bruè, M.~Calzi, G.~E.~Comi, and G.~Stefani},
{\em A distributional approach to fractional Sobolev spaces and fractional variation: asymptotics. II},
C.\ R., Math., Acad.\ Sci.\ Paris {\bf 360}, 589--626, (2022).

\bibitem{BuaGhoGho19}
{\sc L.~Bălilescu, A.~Ghosh, and T.~Ghosh},
{\em $H$-convergence and homogenization of non-local elliptic operators in both perforated and non-perforated domains},
Z.\ Angew.\ Math.\ Phys.\ {\bf 70}, Paper No.\ 171, (2019).

\bibitem{bucval16}
{\sc C.~Bucur and E.~Valdinoci},
Nonlocal diffusion and applications, 
Lecture Notes of the Unione Matematica Italiana 20.
Springer, Unione Matematica Italiana, Bologna, 2016.

\bibitem{cdv19}
{\sc A.~Carbotti, S.~Dipierro, and E.~Valdinoci},
Local density of solutions to fractional equations,
De Gruyter Studies in Mathematics 74.
De Gruyter, Berlin, 2019.

\bibitem{CDMD1990}
{\sc V.~Chiad\`o{} Piat, G.~Dal Maso and A.~Defranceschi},
{\em {$G$}-convergence of monotone operators}, 
Ann. Inst. H. Poincar\'e{} C Anal. Non Lin\'eaire\ {\bf 7}, 123--160, (1990).

\bibitem{ComiStefani19}
{\sc G.~E.~Comi and G.~Stefani},
{\em A distributional approach to fractional Sobolev spaces and fractional variation: existence of blow-up}, 
J.\ Funct.\ Anal.\ {\bf 277}, 3373--3435, (2019).

\bibitem{ComiStefani22}
{\sc G.~E.~Comi and G.~Stefani},
{\em Leibniz rules and Gauss-Green formulas in distributional fractional spaces.} 
J.\ Math.\ Anal.\ Appl.\ {\bf 514}, Paper No.\ 126312, (2022).

\bibitem{ComiStefani23}
{\sc G.~E.~Comi and G.~Stefani},
{\em A distributional approach to fractional Sobolev spaces and fractional variation: asymptotics I.} 
Rev.\ Mat.\ Complut.\ {\bf 36}, 491--569, (2023).

\bibitem{CuKrSc23}
{\sc J.~Cueto, C.~Kreisbeck, and H.~Schönberger},
{\em A variational theory for integral functionals involving finite-horizon fractional gradients},
Fract.\ Calc.\ Appl.\ Anal.\ {\bf 26}, 2001--2056, (2023).

\bibitem{DalMaso}
{\sc G.~Dal Maso},
An introduction to $\Gamma$-convergence, 
Progress in Nonlinear Differential Equations and their Applications. 8. 
Birkhäuser, Basel, 1993.

\bibitem{DalMasoCrack}
{\sc G.~Dal Maso, G.~A.~Francfort, and R.~Toader},
{\em Quasistatic crack growth in nonlinear elasticity},
Arch.\ Ration.\ Mech.\ Anal.\ {\bf 176}, 165--225, (2005).

\bibitem{DGS}
{\sc E.~De Giorgi and S.~Spagnolo},
{\em Sulla convergenza degli integrali dell’energia per operatori ellittici del secondo ordine},
Boll.\ Unione Mat.\ Ital., IV.\ Ser.\ {\bf 8}, 391--411, (1973).

\bibitem{FBRS17}
{\sc J.~Fernández Bonder, A.~Ritorto, and A.~M.~Salort}, 
{\em $H$-convergence result for nonlocal elliptic-type problems via Tartar’s method},
SIAM J.\ Math.\ Anal.\ {\bf 49}, 2387--2408, (2017).

\bibitem{FMT04}
{\sc G.~Francfort, F.~Murat and L.~Tartar}, 
{\em Monotone operators in divergence form with {$x$}-dependent multivalued graphs},
Boll. Unione Mat. Ital. Sez. B Artic. Ric. Mat. (8)\ {\bf 7}, 23--59, (2004).

\bibitem{KPZ19}
{\sc M.~Kassmann, A.~Piatnitski,and E.~Zhizhina},
{\em Homogenization of Lévy-type operators with oscillating coefficients},
SIAM J.\ Math.\ Anal.\ {\bf 51}, 3641--3665, (2019).

\bibitem{KrSc22}
{\sc C.~Kreisbeck and H.~Schönberger},
{\em Quasiconvexity in the fractional calculus of variations: characterization of lower semicontinuity and relaxation},
Nonlinear Anal., Theory Methods Appl., Ser. A, Theory Methods {\bf 215}, Paper No. 112625, (2022).

\bibitem{Maione}
{\sc A.~Maione},
{\em $H$-convergence for equations depending on monotone operators in Carnot groups},
Electron.\ J.\ Differ.\ Equ.\ {\bf 2021}, Paper No.\ 13, (2021).

\bibitem{MPV}
{\sc A.~Maione, F.~Paronetto, and E.~Vecchi},
{\em $G$-convergence of elliptic and parabolic operators depending on vector fields},
ESAIM, Control Optim.\ Calc.\ Var.\ {\bf 29}, Paper No.\ 8, (2023).

\bibitem{MPV2}
{\sc A.~Maione, F.~Paronetto, and S.~Verzellesi},
{\em Variational convergences under moving anisotropies},
submitted paper, arXiv: \href{https://arxiv.org/abs/2504.02552}{2504.02552}, (2025).

\bibitem{MPSC}
{\sc A.~Maione, A.~Pinamonti, and F.~Serra Cassano}, 
{\em $\Gamma$-convergence for functionals depending on vector fields. I: Integral representation and compactness},
J.\ Math.\ Pures Appl.\ {\bf 139}, 109--142, (2020).

\bibitem{MPSC2}
{\sc A.~Maione, A.~Pinamonti, and F.~Serra Cassano}, 
{\em $\Gamma$-convergence for functionals depending on vector fields. II: Convergence of minimizers},
SIAM J.\ Math.\ Anal.\ {\bf 54}, 5761--5791, (2022).

\bibitem{MuratI}
{\sc F.~Murat},
{\em Compacite par compensation},
Ann.\ Sc.\ Norm.\ Super.\ Pisa, Cl.\ Sci., IV.\ Ser.\ {\bf 5}, 489--507, (1978).

\bibitem{MuratII}
{\sc F.~Murat},
{\em Compacite par compensation II},
Recent methods in non-linear analysis, Proc.\ int.\ Meet., Rome 1978, 245--256, (1979).

\bibitem{PatVal15}
{\sc S.~Patrizi and E.~Valdinoci},
{\em Homogenization and Orowan’s law for anisotropic fractional operators of any order},
Nonlinear Anal., Theory Methods Appl., Ser.\ A, Theory Methods {\bf 119}, 3--36, (2015).




\bibitem{Savin}
{\sc O.~Savin and E.~Valdinoci},
{\em $\Gamma$-convergence for nonlocal phase transitions},
Ann.\ Inst.\ Henri Poincaré, Anal.\ Non Linéaire {\bf 29}, 479--500, (2012).

\bibitem{Sbordone}
{\sc C.~Sbordone},
{\em Su alcune applicazioni di un tipo di convergenza variazionale},
Ann.\ Sc.\ Norm.\ Super.\ Pisa, Cl.\ Sci., IV.\ Ser.\ {\bf 2}, 617--638, (1975).

\bibitem{ShiehSpector}
{\sc T.-T.~Shieh and D.~E.~Spector},
{\em On a new class of fractional partial differential equations},
Adv.\ Calc.\ Var.\ {\bf 8}, 321--336, (2015).

\bibitem{ShiehSpectorII}
{\sc T.-T.~Shieh and D.~E.~Spector},
{\em On a new class of fractional partial differential equations II},
Adv.\ Calc.\ Var.\ {\bf 11}, 289--307, (2018).

\bibitem{solci}
{\sc M.~Solci},
{\em Nonlocal-interaction vortices},
SIAM J.\ Math.\ Anal.\ {\bf 56}, 3430--3451, (2024).

\bibitem{Spagnolo}
{\sc S.~Spagnolo},
{\em Sul limite delle soluzioni di problemi di Cauchy relativi all’equazione del calore},
Ann.\ Sc.\ Norm.\ Super.\ Pisa, Sci.\ Fis.\ Mat., III.\ Ser.\ {\bf 21}, 657--699, (1967).



\bibitem{Stein}
{\sc E.~M.~Stein}, 
Singular integrals and differentiability properties of functions,
Princeton Mathematical Series 30. 
Princeton University Press, Princeton, N.J., 1970.


\bibitem{Tartar}
{\sc L.~Tartar},
The general theory of homogenization. A personalized introduction. 
Lecture Notes of the Unione Matematica Italiana 7. 
Springer, Berlin, 2009.

\bibitem{Waurick}
{\sc M.~Waurick},
{\em Nonlocal $H$-convergence},
Calc.\ Var.\ Partial.\ Differ.\ Equ.\ {\bf 57}, Paper No.\ 6, (2018). 

\bibitem{Waurick2}
{\sc M.~Waurick},
{\em Nonlocal $H$-convergence for topologically nontrivial domains}, 
J.\ Funct.\ Anal.\ {\bf 288}, No.\ 3, (2025).


\bibitem{silhavy}
{\sc M.~Šilhavý},
{\em Fractional vector analysis based on invariance requirements (critique of coordinate approaches)},
Contin.\ Mech.\ Thermodyn.\ {\bf 32}, 207--228, (2020).

\end{thebibliography}

\end{document}